\documentclass{amsart}
\usepackage[latin9]{inputenc}
\usepackage{mathrsfs}
\usepackage{amstext}
\usepackage{amsthm}
\usepackage{amssymb}

\makeatletter
\numberwithin{equation}{section}
\numberwithin{figure}{section}
\theoremstyle{plain}
\newtheorem{thm}{\protect\theoremname}
\theoremstyle{plain}
\newtheorem{prop}[thm]{\protect\propositionname}
\theoremstyle{definition}
\newtheorem{example}[thm]{\protect\examplename}
\theoremstyle{plain}
\newtheorem{cor}[thm]{\protect\corollaryname}
\makeatother
\providecommand{\corollaryname}{Corollary}
\providecommand{\examplename}{Example}
\providecommand{\propositionname}{Proposition}
\providecommand{\theoremname}{Theorem}

\input{tcilatex}

\begin{document}
\title{Embedding of topological posets in hyperspaces}
\thanks{We thank professors Gerhard Gierz, Jinlu Li and Nik Weaver for their
insightful comments on an earlier draft of this paper.}
\author{Gerald Beer}
\address{Department of Mathematics, California State University Los Angeles}
\email{gbeer@cslanet.calstatela.edu}
\author{Efe A. Ok}
\address{Department of Economics and Courant Institute of Mathematical
Sciences}
\email{efe.ok@nyu.edu}

\begin{abstract}
We study the problem of topologically order-embedding a given topological
poset $(X,\preceq)$ in the space of all closed subsets of $X$ which is
topologized by the Fell topology and ordered by set inclusion. We show that
this can be achieved whenever $(X,\preceq)$ is a topological semilattice
(resp. lattice) or a topological po-group, and $X$ is locally compact and
order-connected (resp. connected). We give limiting examples to show that
these results are tight, and provide several applications of them. In
particular, a locally compact version of the Urysohn-Carruth metrization
theorem is obtained, a new fixed point theorem of Tarski-Kantorovich type is
proved, and it is found that every locally compact and connected Hausdorff
topological lattice is a completely regular ordered space.
\end{abstract}

\maketitle

\section{Introduction}

Let $(X,\preceq)$ be a poset, which is to say $X$ is a nonempty set and $%
\preceq$ is a partial order on $X.$ We say that $(X,\preceq)$ is a \textit{%
topological poset} if, in addition, $X$ is a topological space and $\preceq$
is a closed subset of $X\times X$.\footnote{%
The literature is not unified in this terminology. Some authors refer to
what we call a topological poset here as a \textit{po-space} or an \textit{%
ordered topological space}. In addition, a partial order $\preceq$ on a
space $X$ that is closed in $X\times X$ is sometimes called\textit{\ closed}%
, sometimes \textit{continuous}, and sometimes \textit{closed-continuous}.}
Since the seminal work of Nachbin \cite{Na}, such structures, in which
topology and order are intimately linked, have been studied extensively. In
addition, they play major roles in a diverse set of applied fields such as
general relativity \cite{M-P,Ming}, rational decision-making \cite{B-M,E-O},
and domain theory \cite{Dom}.

The prototypical example of a poset is $(2^{X},\subseteq)$ where $2^{X}$ is
the power set of $X$ and $\subseteq$ is the set inclusion ordering on $2^{X}$%
. This example is, in fact, universal in the sense that it contains a copy
of $(X,\preceq)$ for any partial order $\preceq$ on $X$. Indeed, there is a
natural way of order-embedding $(X,\preceq)$ into $(2^{X},\subseteq)$ by
associating to each $x\in X$ its \textit{principal ideal }$%
x^{\downarrow}:=\{z\in X:z\preceq x\}.$ Clearly, the partial order axioms
dictate that the mapping $x\mapsto x^{\downarrow}$ from $X$ into $2^{X}$ --
this is called the \textit{canonical order-embedding -- }is injective and
preserves order structures in the sense that $x\preceq y$ iff $%
x^{\downarrow}\subseteq y^{\downarrow}$. When $(X,\preceq)$ is a topological
poset, the canonical order-embedding maps points in $X$ to nonempty closed
subsets of $X.$ The natural codomain for this map is thus $(\mathcal{C}%
(X),\subseteq)$ where $\mathcal{C}(X)$ stands for the set of all closed
subsets of $X.$ In this case we wish the canonical order-embedding to also
preserve the topological structures of the involved posets. This brings up
the problem of finding a suitable topology for $\mathcal{C}(X)$ so that $%
x\mapsto x^{\downarrow}$ is a topological embedding as well.

While this problem seems to have received surprisingly little attention in
the literature, it is certainly not new. A folk theorem of topological order
theory, a version of which was put on record first by Lawson \cite{Law},
says that when $(X,\preceq)$ is a compact Hausdorff topological $\wedge$%
-semilattice, and we endow $\mathcal{C}(X)$ with the classical Vietoris
topology, then $x\mapsto x^{\downarrow}$ is indeed a topological embedding.
(This result will actually obtain as an immediate corollary of one of our
main theorems below.) However, topological posets that arise in applications
are often not compact, so this result does not apply. And the available
embedding results in the noncompact case are not really satisfactory. Misra 
\cite{Misra} extends the above folk theorem to the noncompact case by
endowing $X$ with the order topology and taking $\preceq$ to be total, which
is rather restrictive. By contrast, Choe and Park \cite{ChoePark} work with
a topological poset $(X,\preceq)$ which is a $C$-space.\footnote{%
This means that $\bigcup\nolimits _{x\in S}x^{\downarrow}$ and $%
\bigcup\nolimits _{x\in S}\{z\in X:x\preceq z\}$ are closed for every $S\in%
\mathcal{C}(X).$} Every compact topological poset is indeed a $C$-space, but
unfortunately, the latter property is usually difficult to verify. Besides, 
\cite{ChoePark} uses a coarsening of the Vietoris topology on the codomain
of the canonical order-embedding that depends on $\preceq$. In other words,
the Choe-Park embedding is not universal, in that the topology of the space
in which $(X,\preceq)$ is embedded depends on the partial order $\preceq$.
This limits the applicability of the embedding.

It is plain that a satisfactory solution to the embedding problem at hand
would use a hyperspace topology on $\mathcal{C}(X),$ which is not only
independent of $\preceq$, but is also well-behaved. The immediate candidates
in this regard are, of course, the Vietoris topology and the Hausdorff
metric topology (when $X$ is a metric space), but these topologies become
too demanding when the underlying space $X$ is not compact. For example, it
is only natural that the sequence of lines $\langle L_{n}\rangle$ in the
plane, where $L_{n}$ is defined by the equation $y=nx$, converge to the
vertical axis, but neither the Vietoris topology nor the Hausdorff metric
topology delivers this conclusion.

A weaker, and much better behaved, topology on $\mathcal{C}(X)$ is the 
\textit{Fell topology}. This topology is widely studied since the seminal
contribution of Fell \cite{Fell}, and it is found to have remarkable
properties. In Section 3, we formally define the Fell topology and discuss
its basic features, but we should note right away that this topology is the
same as the Vietoris topology when $X$ is compact and Hausdorff (as well as
the Hausdorff metric topology when $X$ is compact and metrizable). While
convergence with respect the Fell topology may initially appear somewhat
abstract, it is in fact intimately linked to the classical \textit{%
Kuratowski-Painlev� }(K-P) \textit{convergence }\cite{Ku} which is often
easier to visualize. Most importantly for our purposes, when $X$ is first
countable and Hausdorff (in particular, when it is metrizable), the two
convergence notions coincide for sequences of closed sets.\footnote{%
For a comprehensive account of Kuratowski-Painlev� convergence for nets of
closed sets and its relation to the Fell topology, see \cite{Be,KT}.} This
is important because, since the seminal work of Wijsman \cite{Wij}, it
became evident that K-P convergence is fundamental to finite-dimensional
convex analysis. As such, it has also led to the investigation of other
convergence notions and set topologies in infinite dimensions, including
Mosco convergence, Attouch-Wets convergence and the Joly-slice topology.
(See, for instance, \cite{At,Be,Lu}.)

The primary objective of the present paper is to investigate if, and when,
the canonical order-embedding from a topological poset $(X,\preceq)$ into $(%
\mathcal{C}(X),\subseteq)$, where $\mathcal{C}(X)$ is endowed with the Fell
topology, is a topological embedding. In Section 4, we observe that $%
x\mapsto x^{\downarrow}$ need not be continuous in this sense even when $X$
is a metric continuum (so the folk theorem we mentioned above is not valid
for compact topological posets). However, the situation is markedly
different in the case of Hausdorff topological $\wedge$-semilattices. Our
first main result shows that $x\mapsto x^{\downarrow}$ is continuous
whenever $(X,\preceq)$ is such a poset. Since the Fell topology reduces to
the Vietoris topology when $X$ is a compact Hausdorff space, this result
provides a proof for the said folk theorem. However, in the general context
of our continuity theorem, the use of the Fell topology is essential. In
Section 4, we show that the canonical order-embedding may fail to be
continuous if we equip $\mathcal{C}(X)$ with the Vietoris, or the Hausdorff
metric, topology, even when $(X,\preceq)$ is a locally compact metric $\wedge
$-semilattice.

Our second main result, also presented in Section 4, shows that the inverse
of the canonical order-embedding from $\mathcal{C}_{\downarrow}(X):=\{x^{%
\downarrow}:x\in X\}$ onto $X$ is continuous (with respect to the relative
Fell topology), provided that $(X,\preceq)$ is a locally compact topological
poset which has connected order-intervals. Putting these results together
yields our first embedding theorem: \textit{Every locally compact Hausdorff
topological }$\wedge$\textit{-semilattice }$(X,\preceq)$ \textit{with
connected order intervals} \textit{is topologically order-embedded in the
topological poset }$(\mathcal{C}(X),\subseteq)$ \textit{by the canonical map}
$x\mapsto x^{\downarrow}$. But every connected topological lattice has
connected order-intervals. This gives our second embedding theorem: \textit{%
Every locally compact and connected Hausdorff topological lattice }$%
(X,\preceq)$ \textit{can be topologically order-embedded in }$(\mathcal{C}%
(X),\subseteq)$. Put a bit more precisely, $x\mapsto x^{\downarrow}$ is both
a homeomorphism and an order-isomorphism from $X$ onto $\mathcal{C}%
_{\downarrow}(X)$ whenever the latter is such a lattice.

Needless to say, the main advantage of universal embedding theorems is that
they allow for a unified investigation of a variety of mathematical
structures by means of a particular model. Our embedding theorems are useful
precisely in this sense. To illustrate, we provide several applications in
Section 5. First, we use this theorem to extend the Urysohn-Carruth
metrization theorem, which says that every compact metric poset can be
remetrized by a radially convex metric, to the context of topological posets
that are covered by either of our embedding theorems. Second, we use our
positive result on the continuity of the inverse of the canonical
order-embedding to give sufficient topological conditions for a continuous $%
\wedge$-homomorphism between $\wedge$-complete topological $\wedge$%
-semilattices to preserve arbitrary infima. Third, we again use that
positive result to provide a fixed point theorem in which local compactness
and order-connectedness replace the order-homomorphism condition in the
classical Kantorovich-Tarski fixed point theorem. Fourth, we revisit
Anderson's theorem on local order-convexity of locally compact and connected
Hausdorff topological lattices, and provide a very short proof for it.
Finally, we utilize our second embedding theorem to prove that every such
topological lattice is completely regularly ordered (in the sense of
Nachbin).

While our approach in this paper is largely order-theoretic, we should note
that analogous topological order-embedding theorems can also be achieved in
the present setup by replacing the lattice structure with other
order-compatible algebraic structures. In particular, in Section 6, we show
that, where $\mathcal{C}(X)$ is endowed with the Fell topology, \textit{%
every locally compact topological partially ordered group }$(X,\cdot,\preceq)
$ \textit{with connected order intervals} \textit{is topologically
order-embedded in the topological poset }$(\mathcal{C}(X),\subseteq)$ 
\textit{by the canonical order-embedding}.

The paper concludes with the specification of a substantial open problem
that emanates from the present work.

\section{Preliminaries}

\subsection{Posets}

Let $X$ be a nonempty set. By a \textit{partial order} $\preceq$ on $X$, we
mean a reflexive, antisymmetric and transitive binary relation on $X$. As
usual $x\preceq y$ means $(x,y)\in$ $\preceq$, and $x\preceq y\preceq z$
means both $x\preceq y$ and $y\preceq z$ hold, etc.. For any nonempty $%
S\subseteq X,$ we write $x\preceq S$ to mean that $x\preceq y$ for each $%
y\in S$ (in which case we say that $x$ is a $\preceq$\textit{-lower bound}
for $S$)$.$ The expression $S\preceq x$ (which means $x$ is an $\preceq$%
\textit{-upper bound} for $S$) is similarly interpreted. We let $\bigwedge S$
stand for the greatest $\preceq$-lower bound for $S$ (if it exists), but as
usual, write $x\wedge y$ instead of $\bigwedge\{x,y\}.$ In turn, we let $%
S\wedge T:=\{x\wedge y:(x,y)\in S\times T\}$ for any nonempty $S,T\subseteq
X,$ but write $S\wedge y$ instead of $S\wedge\{y\}$ for any $y\in X.$ The
expressions $\bigvee S$, $x\vee y$, $S\vee T,$ and $S\vee y$ are defined
analogously.

As usual, the ordered pair $(X,\preceq)$ is called a \textit{poset} in which
case we call $X$ the \textit{carrier} of $\preceq$. As we have already noted
in Section 1, a \textit{principal ideal} in this poset is a set of the form $%
x^{\downarrow}:=\{z\in X:z\preceq x\}$ where $x\in X.$ Dual to this notion
is that of a \textit{principal filter }which is a set of the form $%
x^{\uparrow}:=\{z\in X:x\preceq z\}$. By an \textit{order-interval} in $%
(X,\preceq),$ we mean a set of the form $x^{\uparrow}\cap y^{\downarrow}$
for some $x,y\in X$ with $x\preceq y.$ For any $S\subseteq X,$ we define the 
$\preceq$\textit{-decreasing closure} of $S$ as $S^{\downarrow}:=\{z\in
X:z\preceq x$ for some $x\in S\}$, and define the $\preceq$\textit{%
-increasing closure} $S^{\uparrow}$ of $S$ dually. In turn, $S$ is said to
be $\preceq$\textit{-decreasing }if $S=S^{\downarrow},$ and $\preceq$\textit{%
-increasing }if $S=S^{\uparrow}$. On the other hand, $S$ is said to be $%
\preceq$\textit{-convex} if $x^{\uparrow}\cap y^{\downarrow}\subseteq S$ for
every $x,y\in S.$

By an \textit{order-embedding }from a poset $(X_{1},\preceq_{1})$ to a poset 
$(X_{2},\preceq_{2})$, we mean a map $\phi:X_{1}\rightarrow X_{2}$ such that 
$x\preceq_{1}y$ iff $\phi(x)\preceq_{2}\phi(y)$. This map is automatically
injective. If, in addition, $\phi$ is surjective, it is then called an 
\textit{order-isomorphism}.

A poset $(X,\preceq)$ is said to be an \textit{inf-semilattice}, or an $%
\wedge$\textit{-semilattice}, if $x\wedge y$ exists for every $x,y\in X,$
and a \textit{complete }$\wedge$\textit{-semilattice }if $\bigwedge S$
exists for every nonempty $S\subseteq X.$ The notions of $\vee$\textit{%
-semilattice} and\textit{\ complete }$\vee$\textit{-semilattice }are defined
dually. A (\textit{complete}) \textit{lattice} is a poset that is at once a
(complete) $\wedge$- and a (complete) $\vee$-semilattice. A \textit{chain} $C
$ in $(X,\preceq)$ is a nonempty subset of $X$ such that for any $x,y\in C$,
either $x\preceq y$ or $y\preceq x$ holds. By an \textit{antichain} $A$ in $%
(X,\preceq),$ we mean a set $A\subseteq X$ such that no two distinct members
of $A$ are comparable with respect to $\preceq$.

\subsection{Topological Posets}

We denote a topological space $(X,\tau)\ $simply as $X$ when there is no
ambiguity about the topology under consideration. The interior, closure, and
boundary of a subset $S$ of $X$ are denoted as int$(S),$ cl$(S),$ and bd$(S)$%
, respectively. We emphasize that $X\times X$ is always endowed with the
product topology throughout the exposition.

A \textit{topological poset} (resp., \textit{metric poset}) is a poset $%
(X,\preceq)$ whose carrier $X$ is a topological (resp., metric) space, and $%
\preceq$ is a closed subset of $X\times X$. The following is a basic, but
useful, characterization of this closedness axiom \cite[Proposition 1]{Na}.

\begin{prop}
\label{closed ord}Let $(X,\preceq)$ be a poset where $X$ is a topological
space. The following conditions are equivalent:

\begin{enumerate}
\item $\preceq$ is a closed subset of $X\times X$;

\item whenever $x\preceq y$ is false, there exist disjoint neighborhoods $V$
of $x$ and $U$ of $y$ such that $V$ is $\preceq$-increasing and $U$ is $%
\preceq$-decreasing.
\end{enumerate}
\end{prop}

It follows readily from this observation that the topology of a topological
poset is always Hausdorff, because if $x$ and $y$ are distinct points in $X$
either $x\preceq y$ or $y\preceq x$ fails (by antisymmetry of $\preceq$). It
is also worth noting that all principal ideals and filters (hence all
order-intervals) in a topological poset are closed. The converse is false.
For instance, take any infinite set $X$ and endow it with the cofinite
topology. Then, $(X,=)$ is not a topological poset, but all of its principal
ideals and filters are closed.

An $\wedge$-semilattice $(X,\preceq)$ where $X$ is a topological space and $%
(x,y)\mapsto x\wedge y$ is a continuous map from $X\times X$ into $X$, is
said to be a \textit{topological }$\wedge$\textit{-semilattice}. If, in
addition, $(X,\preceq)$ is a lattice and $(x,y)\mapsto x\vee y$ is also a
continuous map from $X\times X$ into $X$, we say that $(X,\preceq)$ is a 
\textit{topological lattice. }By a \textit{Hausdorff topological} $\wedge$%
\textit{-semilattice }(\textit{lattice}), we simply mean a topological $%
\wedge$-semilattice (lattice) whose carrier is a Hausdorff space. And by a 
\textit{complete} \textit{Hausdorff topological }$\wedge$\textit{-semilattice%
}, we mean a Hausdorff topological $\wedge$-semilattice which happens to be
a complete $\wedge$-semilattice.

Even a topological $\wedge$-semilattice that satisfies the $T_{1}$-axiom may
not have closed principal ideals, let alone be a topological poset. For
example, if we order $\mathbb{N}$ as $\cdot\cdot\cdot\prec3\prec2\prec1,$
and endow it with the cofinite topology, we end up with a $T_{1}$
topological $\wedge$-semilattice in which no principal ideal is closed.
However, as it is well known, this anomaly does not arise in the presence of
the Hausdorff axiom. We put this on record for completeness of the
exposition.

\begin{prop}
\label{toplat}A topological $\wedge$-semilattice $(X,\preceq)$ is a
topological poset if and only if $X$ is Hausdorff.
\end{prop}

\begin{proof}
We only need to prove the sufficiency. To this end, take any net $%
\langle(x_{\lambda},y_{\lambda})\rangle$ in $\preceq$ that converges to some 
$(x,y)\in X\times X$. Then, $x_{\lambda}\wedge y_{\lambda}\rightarrow
x\wedge y$ by continuity of $\wedge$. But for each $\lambda,\ $we have $%
x_{\lambda}\wedge y_{\lambda}=x_{\lambda}$, whence $x_{\lambda}\wedge
y_{\lambda}\rightarrow x$ as well. As a net can converge to at most one
limit in a Hausdorff space, we thus obtain $x\wedge y=x$, that is, $x\preceq
y$, as desired.
\end{proof}

Following Gierz et al. \cite[Definition VI-5.13]{Dom}, we say that a
topological poset $(X,\preceq)$ is \textit{order-connected} if the order
interval $x^{\uparrow}\cap y^{\downarrow}$ is a connected subset of $X$ for
every $x,y\in X$ with $x\preceq y.$ This property will play an essential
role in what follows. The following observation relates it to the property
of connectedness of principal ideals and filters.

\begin{prop}
Let $(X,\preceq)$ be a topological poset. If this poset is order-connected,
then its principal ideals and filters are connected. If $(X,\preceq)$ is a
topological $\wedge$-semilattice, the converse holds as well.
\end{prop}

\begin{proof}
Assume $(X,\preceq)$ is order-connected, and take any $x\in X.$ Then, $%
x^{\downarrow}=\bigcup\{z^{\uparrow}\cap x^{\downarrow}:z\in
x^{\downarrow}\},$ so $x^{\downarrow}$ is connected, being the union of a
collection of connected sets with a point in common, and similarly for $%
x^{\uparrow}$. Conversely, if $(X,\preceq)$ is a topological $\wedge$%
-semilattice with connected principal filters, it must be order-connected.
This is because, in that case, for any $x,y\in X$ with $x\preceq y,$ we have 
$x^{\uparrow}\cap y^{\downarrow}=x^{\uparrow}\wedge y,$ whence $%
x^{\uparrow}\cap y^{\downarrow}$ is connected, being the continuous image of
the connected set $x^{\uparrow}$.
\end{proof}

The second assertion of this proposition is false for topological posets in
general. For example, let $X$ stand for the boundary of the 2-cell $%
[-1,1]^{2},$ and endow $X$ with the usual topology and the coordinatewise
order. This yields a topological poset with connected principal ideals and
filters, which is not order-connected (as $x^{\uparrow}\cap
y^{\downarrow}=\{x,y\}$ where $x=(0,-1)$ and $y=(0,1)$).

By a \textit{topological order-embedding} from a topological poset $%
(X_{1},\preceq_{1})$ into another topological poset $(X_{2},\preceq_{2})$,
we mean an order embedding $\phi:X_{1}\rightarrow X_{2}$ such that $\phi$ is
a topological embedding (that is, it is a homeomorphism from $X_{1}$ onto $%
\phi(X_{1})$ where the latter is endowed with the subspace topology).

\section{A Crash Course on the Fell Topology}

In this section we provide a quick review of a few essentials of hyperspace
theory with an emphasis on the Fell topology. On the one hand, this makes
the exposition largely self-contained, as most hyperspace concepts that we
utilize subsequently are introduced here. On the other hand, this discussion
aims to underscore the advantages of our embedding theorems that stem from
the remarkable character of the Fell topology. Save for one exception
(Proposition 3.1), all of the results mentioned below can be found in the
monographs \cite{Be} and \cite{KT}.

Let $X$ be a Hausdorff space. In what follows, we denote the collection of
all closed subsets of $X$ by $\mathcal{C}(X)$, and that of all nonempty
closed subsets of $X$ by $\mathcal{C}_{0}(X)$. Formally speaking, by a 
\textit{hyperspace, }we mean any nonempty subfamily of $\mathcal{C}(X)$
equipped with some topology.

There are a variety of interesting ways in which one may topologize $%
\mathcal{C}(X)$ in a manner faithful to the topology of $X$ (in the sense of
rendering $x\mapsto\{x\}$ an embedding). The most well-known of these is the 
\textit{Hausdorff metric topology} which requires $X$ to be a metric space.
This topology is induced by the extended real-valued \textit{Hausdorff
metric }$H$ on $\mathcal{C}(X),$ which is defined for distinct $A$ and $B$
as 
\begin{equation*}
H(A,B):=\max\left\{ \sup_{a\in A}d(a,B),\sup_{b\in B}d(b,A)\right\} . 
\end{equation*}
Here, of course, $d(a,B):=\inf_{b\in B}d(a,b)$ if $B\neq\varnothing,$ and $%
d(a,\varnothing):=\infty$, and similarly for $d(b,A),$ where $d$ is the
metric of the space $X$.

There is a natural way of extending the Hausdorff metric topology to uniform
spaces. (See \cite{Mi} and \cite[p. 250]{Wi}.) Let $\mathscr{B}$ be a base
for a diagonal uniformity $\mathscr{D}$ on $X$ consisting of symmetric sets,
where $X$ is equipped with the (completely regular) topology induced by the
uniformity. For any $S\subseteq X$ and $B\in\mathscr{B}$, put 
\begin{equation*}
B(S):=\{x\in X:(x,s)\in B\text{ for some }s\in S\}. 
\end{equation*}
The $\mathscr{D}$-\textit{Hausdorff uniformity} on $\mathcal{C}(X)$ has as a
base all sets of the form 
\begin{equation*}
U_{B}:=\{(F_{1},F_{2})\in\mathcal{C}(X)\times\mathcal{C}(X):F_{2}\subseteq
B(F_{1})\ \text{and}\ F_{1}\subseteq B(F_{2})\} 
\end{equation*}
where $B$ runs through $\mathscr{B}$. The topology $\tau_{H}$ on $\mathcal{C}%
(X)$ induced by this separated uniformity (which is independent of the
chosen base) is called the \textit{Hausdorff uniform topology} determined by 
$\mathscr{D}$. If the uniformity of $X$ is metric, this topology becomes the
Hausdorff metric topology induced by that metric.

This approach does not extend to the case where $X$ is a space that is not
completely regular. The first serious study of a hyperspace topology, where $%
X$ is any topological space, was conducted in the seminal article of Michael 
\cite{Mi}. While Michael called his topology the \textit{finite topology},
the literature refers to it much more commonly as the \textit{Vietoris
topology}. To describe this topology, we adopt the notation 
\begin{equation*}
S^{-}:=\{A\in\mathcal{C}(X):A\cap S\neq\varnothing\}\text{\hspace{0.2in}and}%
\hspace{0.2in}S^{+}:=\{A\in\mathcal{C}(X):A\subseteq S\} 
\end{equation*}
where $S$ is any subset of $X$ \cite{Be}. The \textit{Vietoris topology} $%
\tau_{V}$ is generated by all sets of the form $O^{-}$ where $O$ is an open
subset of $X,$ and of the form $(X\backslash F)^{+}$ where $F\in\mathcal{C}%
(X)$. When $X$ is regular as well as Hausdorff, $\mathcal{C}(X)$ equipped
with $\tau_{V}$ is a Hausdorff space. In addition, for both the Hausdorff
uniform topology and the Vietoris topology, $\varnothing$ is an isolated
point of $\mathcal{C}(X)$.

A weaker topology on $\mathcal{C}(X)$ is the \textit{Fell topology}, which
we denote by $\tau_{F}$. This topology is generated by all sets of the form $%
O^{-}$ where $O$ is an open set in $X$, and of the form $(X\backslash K)^{+}$
where $K$ is a compact set in $X$. When $X$ is compact, the Fell topology,
the Vietoris topology and the Hausdorff uniform topology all agree. As will
be clarified shortly, the Fell topology is also intimately related to the 
\textit{Kuratowski-Painlev�} (K-P) \textit{convergence} notion for
sequences of sets.

We recall that sequence $\langle A_{n}\rangle$ in $\mathcal{C}(X)$ is said
to \textit{K-P converge} to $A\in\mathcal{C}(X)$ -- in this case we write $%
A_{n}\overset{\text{K-P}}{\rightarrow}A$ -- if (i) for each $a\in A$ there
exists a sequence $\langle a_{n}\rangle$ such that $a_{n}\rightarrow a$ and $%
a_{n}\in A_{n}$ for each $n$; and (ii) for every strictly increasing
sequence $\langle n_{k}\rangle$ of positive integers and any convergent
sequence $\langle a_{n_{k}}\rangle$ with $a_{n_{k}}\in A_{n_{k}}$ for each $k
$, we have $\lim a_{n_{k}}\in A$ \cite{At,Be,Bg,Ku}.

In the sequel, for any Hausdorff space $X,$ we write $\mathcal{C}^{\text{F}%
}(X)$ for the topological space $(\mathcal{C}(X),\tau_{F}),$ and $\mathcal{C}%
_{0}^{\text{F}}(X)$ for the subspace $\mathcal{C}_{0}(X)$ of $\mathcal{C}^{%
\text{F}}(X).$ (In other words, the topology of $\mathcal{C}_{0}^{\text{F}%
}(X)$ is the relative Fell topology.) The following is a list of notable
facts about these spaces whose proofs can be found in \cite{At,Be,KT}.

\begin{itemize}
\item $\mathcal{C}^{\text{F}}(X)$ is compact;

\item $\mathcal{C}_{0}^{\text{F}}(X)=X^{-}$ is an open subset in $\mathcal{C}%
^{\text{F}}(X)$;

\item $x\mapsto\{x\}$ embeds $X$ in $\mathcal{C}^{\text{F}}(X)$ as a closed
subset, from which many important properties that hold in $\mathcal{C}^{%
\text{F}}(X)$ are forced on the base space;

\item When $X$ is first countable, convergence of sequences in $\mathcal{C}^{%
\text{F}}(X)$ is equivalent to their K-P convergence;

\item When $X$ is a uniform space, the Fell topology is coarser than the
Hausdorff uniform topology on $\mathcal{C}(X)$;

\item $\mathcal{C}^{\text{F}}(X)$ is Hausdorff iff $X$ is locally compact;
in this setting, as an open subset of a compact Hausdorff space, $\mathcal{C}%
_{0}^{\text{F}}(X)$ is a locally compact Hausdorff space;

\item If $X$ is locally compact and has a countable base $\mathcal{V}$ for
its topology, then a countable subbase for the Fell topology is $\{V^{-}:V\in%
\mathcal{V}\}\cup\{(X\backslash\text{cl}(V))^{+}:V\in\mathcal{V}\ \text{with}%
\ V\ \text{relatively compact}\}$. In this setting, by the Urysohn
metrization theorem, $\mathcal{C}^{\text{F}}(X)$ is metrizable.

\item If $\left\langle S_{\lambda}\right\rangle $ is a decreasing (resp.,
increasing) net with respect to $\subseteq$, then $S_{\lambda}\rightarrow%
\bigcap S_{\lambda}$ (resp., $S_{\lambda}\rightarrow$ cl$(\bigcup
S_{\lambda})$) relative to $\tau_{F}$.
\end{itemize}

We will always order the members of $\mathcal{C}(X)$ by the set inclusion
ordering. Thus, henceforth, whenever we refer to $\mathcal{C}^{\text{F}}(X)$%
, or $\mathcal{C}_{0}^{\text{F}}(X),$ as a poset, what we mean is $(\mathcal{%
C}^{\text{F}}(X),\subseteq)$, or $(\mathcal{C}_{0}^{\text{F}}(X),\subseteq),$
as the case may be. The next result characterizes exactly when these posets
are topological.

\begin{prop}
\label{whenposet} Let $X$ be a Hausdorff space. The following conditions are
equivalent:

\begin{enumerate}
\item $X$ is locally compact;

\item $\subseteq$ is closed in $\mathcal{C}(X)\times\mathcal{C}(X)$ in the
Fell topology;

\item $\subseteq$ is closed in $\mathcal{C}_{0}(X)\times\mathcal{C}_{0}(X)$
in the relative Fell topology.
\end{enumerate}
\end{prop}

\begin{proof}
(2) implies (3) trivially. On the other hand, if (3) holds, then $\mathcal{C}%
_{0}(X)$ is Hausdorff in the relative Fell topology (Proposition 2.3), so $X$
must be locally compact \cite[Proposition 5.1.2]{Be}. To complete the proof,
assume that $X$ is locally compact, and take any two nets $\langle
A_{\lambda}\rangle$ and $\langle B_{\lambda}\rangle$ in $\mathcal{C}(X)$
with $A_{\lambda}\subseteq B_{\lambda}$ for each index $\lambda$. Suppose
these nets converge to $A$ and $B$, respectively, relative to the Fell
topology. To derive a contradiction, say there is a point $a$ in $%
A\backslash B.$ Since $B$ is closed and $X$ is locally compact, there is a
compact neighborhood $K$ of $a$ in $X$ such that $K\cap B=\varnothing$. But
this means that $B\in(X\backslash K)^{+}$, so $\langle B_{\lambda}\rangle$
must eventually be in $(X\backslash K)^{+}$. As $\langle A_{\lambda}\rangle$
lies in $\text{int}(K)^{-}$ eventually, it follows that $A_{\lambda}%
\subseteq B_{\lambda}$ fails for some $\lambda$, a contradiction. Thus, $%
A\subseteq B,$ and we are done.
\end{proof}

The reader is invited to check that $(\mathcal{C}(X),\subseteq)$ equipped
with the Vietoris topology is a topological poset, provided that $X$ is
regular. On the other hand, when $X$ is a uniform space, endowing $(\mathcal{%
C}(X),\subseteq)$ with the Hausdorff uniform topology always yields a
topological poset.

\section{Bicontinuity of the Canonical Order-Embedding}

\subsection{Continuity of $x\mapsto x^{\downarrow}$}

We begin with giving a simple, if a bit surprising, example that shows that
the canonical order-embedding from a topological poset $(X,\preceq)$ into $%
\mathcal{C}^{\text{F}}(X)$ need not be continuous, even when $X$ is a metric
continuum.

\begin{example}
\label{start}Consider $X:=([-1,0]\times\{0\})\cup(\{0\}\times\lbrack-1,0])$
as a metric subspace of the Euclidean plane; $X$ is then a compact and
connected metric space. Moreover, $(X,\preceq)$ is a metric poset, where $%
\preceq$ is the coordinatewise order on $X.$ Now notice that $(-\frac{1}{n}%
,0)\rightarrow(0,0)$ and $(0,-1)\in(0,0)^{\downarrow},$ but $(-\frac{1}{n}%
,0)^{\downarrow}$ fails to hit the open ball of radius $\frac{1}{2}$ around $%
(0,-1)$ in $X$ for each $n$. It follows that the canonical order-embedding $%
x\mapsto x^{\downarrow}$ is not continuous from $X$ into $\mathcal{C}^{\text{%
F}}(X)$.\footnote{%
This does not mean that $(X,\preceq)$ cannot be topologically order-embedded
in $\mathcal{C}^{\text{F}}(X)$ in this particular setting. Indeed, $%
(a,b)\mapsto(a,b)^{\downarrow}\cup(b,a)^{\downarrow}$ embeds $(X,\preceq)$
in $\mathcal{C}^{\text{F}}(X)$ topologically as well as order-theoretically.}
\end{example}

If we adjoint $(-1,-1)$ to $X$ in this example, we obtain a lattice which is
in fact a compact Hausdorff topological $\vee$-semilattice, and yet
continuity of $x\mapsto x^{\downarrow}$ still fails. However, somewhat
surprisingly, this anomaly does not arise if the underlying poset is a
Hausdorff topological $\wedge$-semilattice. This is the content of our next
finding.

\begin{thm}
\label{cont}Let $(X,\preceq)$ be a Hausdorff topological $\wedge$%
-semilattice. Then the map $x\mapsto x^{\downarrow}$ from $X$ into $\mathcal{%
C}^{\text{F}}(X)$ is continuous.
\end{thm}

\begin{proof}
By Proposition \ref{toplat}, $X$ is a topological poset, and hence the range
of the order-embedding $x\mapsto x^{\downarrow}$ is indeed contained in $%
\mathcal{C}(X)$. Now fix an arbitrary $x\in X.$ Take any open set $O$ in $X$
with $x^{\downarrow}\in O^{-}$, that is, $x^{\downarrow}\cap O\neq\varnothing
$. We pick any $z$ in this intersection, and note that $x\wedge z=z$ and $O$
is an open neighborhood of $z$ in $X.$ By continuity of $\wedge,$ therefore,
there exist open subsets $U$ and $V$ of $X$ such that $(x,z)\in U\times V$
and $U\wedge V\subseteq O.$ It follows that $y\wedge z\in y^{\downarrow}\cap
O$ for every $y\in U.$ Thus, for each $y\in U,$ $y^{\downarrow}\cap
O\neq\varnothing$, that is, $y^{\downarrow}\in O^{-}$, which means $%
\{y^{\downarrow}:y\in U\}\subseteq O^{-}$.

Next, take any compact subset $K$ of $X$ with $x^{\downarrow}\in(X\backslash
K)^{+}$, that is, $x^{\downarrow}\subseteq X\backslash K$. We wish to find
an open neighborhood $V$ of $x$ in $X$ such that $\{y^{\downarrow}:y\in
V\}\subseteq(X\backslash K)^{+}$. To derive a contradiction, suppose there
is no such $V,$ and let $\mathcal{V}(x)$ stand for the family of all open
neighborhoods of $x$ in $X$. Then, for every $V\in\mathcal{V}(x),$ there is
an $x_{V}\in V$ such that $x_{V}^{\downarrow}$ is not contained within $%
X\backslash K,$ that is, $x_{V}^{\downarrow}\cap K\neq\varnothing$. We
choose any $y_{V}$ in $x_{V}^{\downarrow}\cap K$ for each $V\in\mathcal{V}(x)
$ to form a net $\langle y_{V}\rangle$ in $X$ with the underlying directed
set $(\mathcal{V}(x),\supseteq).$ By compactness of $K,$ this net has a
cluster point, say $y,$ in $K.$ It follows that $(y,x)$ belongs to $\text{cl}%
(\{(y_{V},x_{V}):V\in\mathcal{V}(x)\})$. As $\preceq$ is closed in $X\times
X,$ and $y_{V}\preceq x_{V}$ for each $V\in\mathcal{V}(x),$ we thus conclude
that $y\preceq x.$ But then $y$ belongs to $x^{\downarrow}\cap K,$
contradicting $x^{\downarrow}\subseteq X\backslash K$.
\end{proof}

When $X$ is a compact Hausdorff space, the Vietoris and Fell topologies on $%
\mathcal{C}(X)$ coincide. Since a continuous injection from a compact space
into a Hausdorff space is a topological embedding, the following thus
obtains as an immediate consequence of Theorem \ref{cont}. As we have noted
in Section 1, this is a folk theorem of topological order theory.

\begin{cor}
The canonical order-embedding from a compact Hausdorff topological $\wedge$%
-semilattice $(X,\preceq)$ into $(\mathcal{C}(X),\subseteq)$ is a
topological embedding, where $\mathcal{C}(X)$ is endowed with the Vietoris
topology.
\end{cor}

It is natural to inquire if one can replace the Fell topology with the
Vietoris topology in Theorem \ref{cont}, or at least with the Hausdorff
metric topology when $(X,\preceq)$ is a metric $\wedge$-semilattice. We
conclude this section with two examples that show that the answers to these
questions are negative. In each case, we use a metric $\wedge$-semilattice $%
(X,\preceq)$ where $X\subseteq\mathbb{R}^{2},$ the metric is the Euclidean
metric and $\preceq$ is simply the coordinatewise ordering. In the case of
the Vietoris topology, we can even take $X=\mathbb{R}^{2}$.

\begin{example}
In $\mathbb{R}^{2},$ we have $(\frac{1}{n},0)\rightarrow(0,0),$ but $%
\left\langle (\frac{1}{n},0)^{\downarrow}\right\rangle $ does not converge
to $(0,0)^{\downarrow}$ relative to the Vietoris topology. To see this, put $%
F:=\{(\alpha,\beta)\in\mathbb{R}^{2}:\beta\leq\ln\alpha\},$ and notice that $%
(0,0)^{\downarrow}\in(X\backslash F)^{+}$ while $(\frac{1}{n},0)^{\downarrow}
$ does not belong to $(X\backslash F)^{+}$ for any $n\in\mathbb{N}$.%
\footnote{%
By contrast, $(1,\tfrac{1}{n})^{\downarrow}\rightarrow(1,0)^{\downarrow}$ in
the Fell topology. Indeed, as $\langle(1,\tfrac{1}{n})^{\downarrow}\rangle$
is a decreasing sequence in $\mathcal{C}(X),$ we have $(1,\tfrac{1}{n}%
)^{\downarrow}\overset{\text{K-P}}{\rightarrow}\bigcap_{n\geq1}(1,\tfrac{1}{n%
})^{\downarrow}=(1,0)^{\downarrow}$ in concert with Theorem \ref{cont}.}
\end{example}

The situation with the Hausdorff metric topology is less straightforward.
For one thing, in the plane equipped with the Euclidean metric, we have $%
H(x^{\downarrow},y^{\downarrow})\leq\left\Vert x-y\right\Vert _{2}$ for any $%
x,y\in\mathbb{R}^{2}$, so we cannot work with the entire plane to produce
the desired counterexample. This is partly responsible for the intricacy of
the next construction.

\begin{example}
In $\mathbb{R}^{2}$, let $A:=\{(-1,-k):k\in\mathbb{N}\},$ let $%
B:=\{(0,-k):k\in\mathbb{N}\}$ and for each integer $m\geq2$, put $C_{m}:=\{(-%
\frac{1}{m},-k):k=1,...,m\}.$ The carrier of our topological poset is 
\begin{equation*}
X:=A\cup B\cup C_{2}\cup C_{3}\cup\cdot\cdot\cdot 
\end{equation*}
which we view as a metric subspace of $\mathbb{R}^{2}$ and endow with the
coordinatewise order $\preceq$. Clearly, $(X,\preceq)$ is a complete,
locally compact and totally disconnected metric space whose only
non-isolated points lie in $B.$ Relative to $\preceq,$ we have $H((-\frac{1}{%
n},-1)^{\downarrow},(0,-1)^{\downarrow})=1$ for every $n\geq2,$ so the
canonical order-embedding would not be continuous if we endowed $\mathcal{C}%
(X)$ with the Hausdorff metric. Besides, it is easily seen that $(X,\preceq)$
is a lattice. In what follows, we prove that $(X,\preceq)$ is actually a
metric $\wedge$-semilattice.

To prove that $\wedge$ is a continuous map from $X\times X$ onto $X,$ it is
enough to focus only on the subdomain $X\times B$, because $\wedge$ is a
symmetric map and all points in $X\backslash B$ are isolated. We may also
reason sequentially, as the carrier space is metric.
\end{example}

Take any $x\in X$ and put $y:=(0,-k)$ for some $k\in\mathbb{N}$. Let $%
(x_{n},y_{n})$ be a sequence in $X\times X$ with $x_{n}\rightarrow x$ and $%
y_{n}\rightarrow y.$ In the sequel, we will write $x_{n}:=(\alpha_{n},%
\beta_{n})$ and $y_{n}:=(\gamma_{n},\mu_{n})$ for each $n\in\mathbb{N}$.

First suppose that $x\in B$ as well, that is, $x=(0,-l)$ for some $l\in%
\mathbb{N}$. This means that $\left\langle \alpha_{n}\right\rangle $ and $%
\left\langle \gamma_{n}\right\rangle $ converge to $0$ and that, eventually, 
$\beta_{n}:=-l$ and $\mu_{n}=-k.$ In view of the continuity of the minimum
functional on $\mathbb{R}^{2},$ therefore, 
\begin{equation*}
\lim_{n\rightarrow\infty}x_{n}\wedge
y_{n}=\lim_{n\rightarrow\infty}(\min\{\alpha_{n},\gamma_{n}\},\min\{-l,-k%
\})=(0,\min\{-l,-k\})=x\wedge y. 
\end{equation*}
Next suppose that $x\in X\backslash B,$ and put $x:=(\alpha,\beta).$ In this
case, $x$ is isolated, so we may assume $x_{n}=x$ for all $n.$ As $%
\gamma_{n}\rightarrow0>\alpha$ and $\mu_{n}\rightarrow-k,$ we can also
assume that $\gamma_{n}>\alpha$ and $\mu_{n}=-k$ for all $n$. We now claim
that with these simplifying assumptions, $\left\langle x_{n}\wedge
y_{n}\right\rangle =\left\langle x\wedge y_{n}\right\rangle $ is a constant
sequence. This is verified by considering five separate cases: 
\begin{equation*}
x\wedge y_{n}=\left\{ 
\begin{array}{ll}
x, & \text{if }\alpha=-1\text{ and }\beta\leq-k, \\ 
(-1,-k), & \text{if }\alpha=-1\text{ and }\beta>-k, \\ 
(-1,-k), & \text{if }-1<\alpha<-1/k, \\ 
(\alpha,-k), & \text{if }-1/k\leq\alpha\text{ and }\beta\geq-k, \\ 
x, & \text{if }-1/k\leq\alpha\text{ and }\beta<-k%
\end{array}%
\right. 
\end{equation*}
for every $n\in\mathbb{N}$. It follows $x\wedge y_{n}=x\wedge y$ for each $n$
in every contingency, and we may thus conclude that $x_{n}\wedge
y_{n}\rightarrow x\wedge y$ when $x\in X\backslash B$ as well.

\subsection{Continuity of the Inverse of $x\mapsto x^{\downarrow}$}

We again start with some limiting examples that show that the inverse of the
canonical order-embedding need not be continuous without suitable hypotheses.

\begin{example}
Let $X:=(-\infty,0)^{2}\cup\{\mathbf{0}\}\cup(0,\infty)^{2}$ which we view
as a subspace of $\mathbb{R}^{2}$ relative to the usual topology. (Here $%
\mathbf{0}$ stands for the origin of $\mathbb{R}^{2}$.) Then, $(X,\preceq),$
where $\preceq$ is the usual coordinatewise order, is a Hausdorff
topological lattice. However, the map $x^{\downarrow}\mapsto x$ is not
continuous at $\mathbf{0}^{\downarrow}$. After all, here we have $(\frac{1}{n%
},1)^{\downarrow}\overset{\text{K-P}}{\rightarrow}\mathbf{0}^{\downarrow}$,
whence $(\frac{1}{n},1)^{\downarrow}\rightarrow\mathbf{0}^{\downarrow}$ in
the Fell topology.
\end{example}

\begin{example}
Let $\mathbf{0}$ again stand for the 2-vector $(0,0),$ and put $x_{n}:=(%
\frac{1}{n},n)$ for each $n\in\mathbb{N}$. We now consider the poset $%
(X,\preceq)$ where $X:=\{\mathbf{0},x_{1},x_{2},...\}$ and $\preceq$ is the
coordinatewise order. We endow $X$ with the discrete topology (which is the
subspace topology $X$ inherits from the Euclidean topology of the plane).
Then, $(X,\preceq)$ is a locally compact Hausdorff topological $\wedge$%
-semilattice. However, the inverse of the canonical order-embedding in this
setting is not continuous. Indeed, we have $x_{n}^{\downarrow}=\{\mathbf{0}%
,x_{n}\}\overset{\text{K-P}}{\rightarrow}\{\mathbf{0}\}$ here, so $%
x_{n}^{\downarrow}\rightarrow\mathbf{0}^{\downarrow}$ in the Fell topology,
but $\langle x_{n}\rangle$ does not converge to $\mathbf{0}$.
\end{example}

In the first example above, the difficulty emanates from the lack of local
compactness (at $\mathbf{0}$), and in the second one, because
order-intervals are not connected. The next result shows that if the
underlying topological poset $(X,\preceq)$ is locally compact and possesses
connected order-intervals, then the inverse of the associated canonical
order-embedding is indeed continuous on $\mathcal{C}_{\downarrow}(X):=\{x^{%
\downarrow}:x\in X\},$ the set of all principal ideals of $(X,\preceq)$.

\begin{thm}
\label{continv}Let $(X,\preceq)$ be a locally compact and order-connected
topological poset. Then the map $x^{\downarrow}\mapsto x$ from $\mathcal{C}%
_{\downarrow}(X)$ onto $X$ is continuous (with respect to the relative Fell
topology).
\end{thm}

\begin{proof}
Let $\langle x_{\lambda}\rangle$ be a net in $X$ (with the underlying
directed set $(\Lambda,\trianglelefteq)$), and $x$ a point in $X$, such that 
$x_{\lambda}^{\downarrow}\rightarrow x^{\downarrow}$ in the Fell topology.
We need to show that $x_{\lambda}\rightarrow x$.

To derive a contradiction, suppose $\langle x_{\lambda}\rangle$ does not
converge to $x$ in $X.$ As $X$ is locally compact and Hausdorff, there is a
local base of compact neighborhoods of $x$ in $X$ \cite[p. 238]{Du}. By
replacing $\Lambda$ by a suitable cofinal subset if necessary, we may thus
assume that there exists a compact neighborhood $K$ of $x$ in $X$ such that $%
x_{\lambda}\in X\backslash K$ for every $\lambda\in\Lambda.$

We now define an auxiliary net of points in $X$ that converges to $x.$ Let $%
\mathcal{O}_{K}(x)$ stand for the family of all open neighborhoods $O$ of $x$
in $X$ such that $O\subseteq$ int$(K).$ Put 
\begin{equation*}
\Omega:=\left\{ (\lambda,O):\lambda\in\Lambda,\text{ }O\in\mathcal{O}_{K}(x)%
\text{ and }x_{\lambda}^{\downarrow}\cap O\neq\varnothing\right\} \text{,} 
\end{equation*}
and consider the partial order $\blacktriangleleft$ on $\Omega$ defined by 
\begin{equation*}
(\alpha,U)\blacktriangleleft(\beta,V)\text{\hspace{0.2in}iff\hspace{0.2in}}%
\alpha\trianglelefteq\beta\text{ and }V\subseteq U\text{.} 
\end{equation*}
Clearly, $(\Omega,\blacktriangleleft)$ is a directed poset. For each $%
(\lambda,O)\in\Omega,$ we pick any $w_{\lambda,O}\in
x_{\lambda}^{\downarrow}\cap O.$ As $x_{\lambda}^{\downarrow}\rightarrow
x^{\downarrow}$ in the Fell topology, for every open neighborhood $U$ of $x$
in $X,$ there is a $\lambda_{U}\in\Lambda$ such that 
\begin{equation*}
x_{\lambda}^{\downarrow}\in(U\cap\text{int}(K))^{-}\text{\hspace{0.2in}%
whenever\hspace{0.2in}}\lambda_{U}\trianglelefteq\lambda\text{,} 
\end{equation*}
whence 
\begin{equation*}
w_{\lambda,O}\in O\subseteq U\text{\hspace{0.2in}whenever\hspace{0.2in}}%
(\lambda_{U},U\cap\text{int}(K))\blacktriangleleft(\lambda,O)\text{.} 
\end{equation*}
Thus: $w_{\lambda,O}\rightarrow x.$

For any $\lambda\in\Lambda,$ the connectedness of $w_{\lambda,O}^{\uparrow}%
\cap x_{\lambda}^{\downarrow}$ entails that there is a $z_{\lambda,O}\in$ bd$%
(K)\cap w_{\lambda,O}^{\uparrow}\cap x_{\lambda}^{\downarrow}$ for,
otherwise, int$(K)$ and $X\backslash K$ would determine a nontrivial
separation of $w_{\lambda,O}^{\uparrow}\cap x_{\lambda}^{\downarrow}$. As bd$%
(K)$ is compact, being a closed subset of a compact set, the net $\langle
z_{\lambda,O}\rangle$ must have a cluster point $z$ in bd$(K).$ The proof
will be completed by showing that (i) $z\preceq x$ and (ii) $x\preceq z,$
hence $z=x$ (which contradicts $x$ belonging to int$(K)$).

Suppose (i) fails, that is, $z$ lies outside of $x^{\downarrow}$. Since $X$
is locally compact, $z$ must then have a compact neighborhood $C$ which is
disjoint from $x^{\downarrow}$. As $x_{\lambda}^{\downarrow}\rightarrow
x^{\downarrow}$ in the Fell topology, there is a $\lambda_{0}\in\Lambda$
such that $x_{\lambda}^{\downarrow}\cap C=\varnothing$ for every $%
\lambda\in\Lambda$ with $\lambda_{0}\trianglelefteq\lambda$. On the other
hand, since $z$ is a cluster point of $\langle z_{\lambda,O}\rangle,$ we
have $z_{\lambda,O}\in$ int$(C)$ for a cofinal set of indices $(\lambda,O)$
in $\Omega$. As $z_{\lambda,O}\in x_{\lambda}^{\downarrow}$ for all $\lambda,
$ this is a contradiction. Conclusion: $z\preceq x$.

Given that $\preceq$ is closed, (ii) follows from the fact that (a) $%
w_{\lambda,O}\rightarrow x;$ (b) $z$ is a cluster point of $\langle
z_{\lambda,O}\rangle$; and (c) $w_{\lambda,O}\preceq z_{\lambda,O}$ for
every $(\lambda,O)\in\Omega$. To be more precise, take any $V\in\mathcal{O}%
_{K}(x)$ and any open neighborhood $U$ of $z$. Since $x_{\lambda}^{%
\downarrow}\rightarrow x^{\downarrow}$ in the Fell topology, there exists $%
\lambda_{0}\in\Lambda$ such that $x_{\lambda}^{\downarrow}\cap
V\neq\varnothing,$ that is, $(\lambda,V)\in\Omega,$ for every $%
\lambda\in\Lambda$ with $\lambda_{0}\trianglelefteq\lambda$. As $z$ is a
cluster point of $\langle z_{\lambda,O}\rangle,$ there certainly exists a $%
(\lambda,W)\in\Omega$ with $\lambda_{0}\trianglelefteq\lambda$ where both $%
W\subseteq V$ and $z_{\lambda,W}\in U$ hold. Since $w_{\lambda,W}\in W$ as
well, we obtain 
\begin{equation*}
(w_{\lambda,W},z_{\lambda,W})\in W\times U\subseteq V\times U\text{.} 
\end{equation*}
As $w_{\lambda,W}\preceq z_{\lambda,W},$ this argument proves that $(x,z)$
belongs to the closure of $\preceq$ in $X\times X$. Since $\preceq$ is
closed, therefore, $x\preceq z,$ and the proof is complete.
\end{proof}

\subsection{Embedding Theorems}

Putting Theorems \ref{cont} and \ref{continv} together yields our main
embedding result.

\begin{thm}
\label{embed1}Let $(X,\preceq)$ be a locally compact and order-connected
Hausdorff topological $\wedge$-semilattice. Then, the map $x\mapsto
x^{\downarrow}$ topologically order-embeds $(X,\preceq)$ in $\mathcal{C}^{%
\text{F}}(X)$.
\end{thm}

For any topological poset $(X,\preceq),$ let $\mathcal{C}_{\downarrow}^{%
\text{F}}(X)$ stand for $(\mathcal{C}_{\downarrow}(X),\tau_{F}),$ that is,
the space of all principal ideals of $(X,\preceq)$ endowed with the relative
Fell topology. Theorem \ref{embed1} entails that $X$ can be identified with $%
\mathcal{C}_{\downarrow}^{\text{F}}(X)$ topologically, and $(X,\preceq)$
with $(\mathcal{C}_{\downarrow}^{\text{F}}(X),\subseteq)$
order-theoretically, by means of the same morphism. It follows that, in the
context of this theorem, $(\mathcal{C}_{\downarrow}^{\text{F}}(X),\subseteq)$
is a locally compact and order-connected Hausdorff topological $\wedge$%
-semilattice. In particular, even though the $\cap$ operation on $\mathcal{C}%
^{\text{F}}(X)$ is not continuous (even when $X$ is a metric continuum), it
is continuous on $\mathcal{C}_{\downarrow}^{\text{F}}(X)$. To see this,
define $\varphi:X\rightarrow\mathcal{C}_{\downarrow}^{\text{F}}(X)$ by $%
\varphi(x):=x^{\downarrow}$. Then, if $\left\langle x_{\lambda}\right\rangle 
$ and $\left\langle y_{\lambda}\right\rangle $ are two nets in $X$ with $%
x_{\lambda}^{\downarrow}\rightarrow x^{\downarrow}$ and $y_{\lambda}^{%
\downarrow}\rightarrow y^{\downarrow}$ for some $x,y\in X,$ we have 
\begin{equation*}
x_{\lambda}\wedge
y_{\lambda}=\varphi^{-1}(x_{\lambda}^{\downarrow})\wedge\varphi^{-1}(y_{%
\lambda}^{\downarrow})\rightarrow\varphi^{-1}(x^{\downarrow})\wedge%
\varphi^{-1}(y^{\downarrow})=x\wedge y 
\end{equation*}
because $\varphi^{-1}$ and $\wedge$ are continuous. As $(a\wedge
b)^{\downarrow}=a^{\downarrow}\cap b^{\downarrow}$ in the context of any $%
\wedge$-semilattice, therefore, 
\begin{equation*}
x_{\lambda}^{\downarrow}\cap y_{\lambda}^{\downarrow}=(x_{\lambda}\wedge
y_{\lambda})^{\downarrow}=\varphi(x_{\lambda}\wedge
y_{\lambda})\rightarrow\varphi(x\wedge y)=(x\wedge
y)^{\downarrow}=x^{\downarrow}\cap y^{\downarrow} 
\end{equation*}
in view of the continuity of $\varphi$.

If only to highlight the nontrivial nature of this fact, we next present an
example that shows that even when $(X,\preceq)$ is a compact, connected, and
order-connected topological poset which happens to be an $\wedge$%
-semilattice, the intersection operation on $\mathcal{C}_{\downarrow}^{\text{%
F}}(X)$ does not have to be (even separately) continuous.

\begin{example}
Let $\mathbf{0}$ stand for the origin of the Hilbert space $\ell_{2},$ and
consider the sequence $a_{0}\in\ell_{2}$ where $a_{0,2k-1}:=0$ and $%
a_{0,2k}:=\frac{1}{k}$ for every $k\in\mathbb{N}$. Next, for each positive
integer $n,$ let $a_{n}\in\ell_{2}$ be the perturbation of $a_{0}$ defined
by $a_{n,2n-1}:=\frac{1}{n},$ $a_{n,2n}:=0,$ and $a_{n,k}:=a_{0,k}$
otherwise. Since $\left\Vert a_{n}-a_{0}\right\Vert _{2}=\frac{\sqrt{2}}{n}%
\rightarrow0,$ the set $\{a_{0},a_{1},...\}$ is compact. We put $C_{n}:=%
\text{conv}(\{\mathbf{0},a_{n}\})$ for each nonnegative integer $n,$ and
define $X:=C_{0}\cup C_{1}\cup\cdot\cdot\cdot$ which we consider as a
(metric) subspace of $\ell_{2}.$ As $X$ is the continuous image of the
compact set $\{a_{0},a_{1},...\}\times[0,1],$ it is compact. Moreover, being
the union of a collection of convex subsets of $\ell_{2}$ with a point in
common, $X$ is connected.

Clearly, $(X,\preceq)$ is a topological poset, where $\preceq$ is the
coordinatewise order. By construction, for any $x,y\in X$, we have $x\preceq
y$ iff $x=\alpha a_{n}$ and $y=\beta a_{n}$ for some $n\geq0$ and $%
\alpha,\beta\in\lbrack0,1]$ with $\alpha\leq\beta$. It follows that the
order-intervals of $(X,\preceq)$ are convex subsets of $\ell_{2},$ so $%
(X,\preceq)$ is order-connected. Finally, while $\preceq$ is total on $C_{n}$
for any $n\geq0,$ if $x\in C_{n}$ and $y\in C_{m}$ for distinct $m,n\geq0,$
we have $x\wedge y=\mathbf{0}.$ Thus $(X,\preceq)$ is an $\wedge$%
-semilattice as well.

Now, notice that $a_{n}\rightarrow a_{0}$ implies $H(a_{n}^{%
\downarrow},a_{0}^{\downarrow})\rightarrow0,$ where $H$ is the Hausdorff
metric. As the Fell topology is coarser than the Hausdorff metric topology
on $\mathcal{C}(\ell_{2}),$ it follows that $a_{n}^{\downarrow}\rightarrow
a^{\downarrow}$ in the Fell topology. But $a_{n}^{\downarrow}\cap
a_{0}^{\downarrow}=\{\mathbf{0}\}$ for each $n>0$ while $a_{0}^{\downarrow}%
\cap a_{0}^{\downarrow}=a_{0}^{\downarrow}\neq\{\mathbf{0}\}$. Conclusion: $%
\cap$ is not (separately) continuous on $\mathcal{C}_{\downarrow}^{\text{F}%
}(X)$.
\end{example}

Order enters to the hypotheses of Theorem \ref{embed1} from the channels of
the semilattice property and order-connectedness. If we strengthen the
former property to being a topological lattice, then we can replace the
latter channel with a purely topological one. This is our second embedding
theorem.

\begin{thm}
\label{embed2}Let $(X,\preceq)$ be a locally compact and connected Hausdorff
topological lattice. Then, the map $x\mapsto x^{\downarrow}$ topologically
order-embeds $(X,\preceq)$ in $\mathcal{C}^{\text{F}}(X)$.
\end{thm}

\begin{proof}
Take any $x,y\in X$ with $x\preceq y,$ and consider the map $%
\eta:X\rightarrow X$ defined by $\eta(z):=(z\vee x)\wedge y.$ Then, $\eta$
is continuous (as $\wedge$ and $\vee$ are continuous) and $x^{\uparrow}\cap
y^{\downarrow}=\eta(X),$ so $x^{\uparrow}\cap y^{\downarrow}$ is connected.
It follows that $(X,\preceq)$ is order-connected, and Theorem \ref{embed1}
applies.
\end{proof}

\section{Applications}

In this section we provide five applications of the above theorems to
various topics in topological order theory.

\subsection{On Radially Convex Metrization}

Let $(X,\preceq)$ be a poset. We recall that a metric $d$ on $X$ is \textit{%
radially convex }(relative to $\preceq$) if 
\begin{equation*}
x\preceq y\preceq z\hspace{0.2in}\text{implies\hspace{0.2in}}%
d(x,z)\geq\max\{d(x,y),d(y,z)\}. 
\end{equation*}
This concept builds a tight connection between the order and metric
structures that may be imposed on a given set. For example, every order
interval of a poset that is endowed with a radially convex metric is bounded
with respect to that metric. It is thus of interest when one can remetrize a
given metric poset by means of a radially convex metric (without changing
the topology). The major result in this regard, due to Carruth \cite{Carruth}%
, is the following:

\bigskip{}

\noindent \textbf{The Urysohn-Carruth Metrization Theorem.} \textit{Every
compact metric poset can be equivalently remetrized by a radially convex
metric.}

\bigskip{}

The classical \textit{Nachbin extension theorem} says that if $(X,\preceq)$
is a compact topological poset and $S$ a nonempty closed subset of $X$, then
every order-preserving and continuous real map on $S$ can be extended to $X$
in an order-preserving and continuous manner.\footnote{%
This result obtains by putting Theorems 4 and 6 of Chapter 2 of Nachbin \cite%
{Na}.} Carruth \cite{Carruth} uses this result to build a continuous
order-embedding from $X$ into the coordinatewise ordered Hilbert cube. This
bit can be generalized to the context of locally compact and separable
metric posets. But while Carruth's order-embedding is automatically a
topological embedding (because $X$ is compact), it is not clear if this is
so in the more general case as well. Indeed, it is presently unknown if the
Urysohn-Carruth metrization theorem is valid for locally compact and
separable metric posets. However, by using the embedding established in
Theorem \ref{embed1}, we get a positive result to this query in the case of
order-connected metric semilattices.

Before we state the main result of this section, we recall that, given any
metric space $(X,d)$, the \textit{Wijsman topology} is the weak topology on $%
\mathcal{C}_{0}(X)$ induced by the family $\{d(x,\cdot):x\in X\}$ \cite%
{Be,LL}. Thus, a net $\left\langle S_{\lambda}\right\rangle $ in $\mathcal{C}%
_{0}(X)$ converges to an $S\in\mathcal{C}_{0}(X)$ relative to the Wijsman
topology iff $d(x,S_{\lambda})\rightarrow d(x,S)$ for each $x\in X.$

\begin{thm}
\label{U-C}Every locally compact, second-countable and order-connected
Hausdorff topological $\wedge$-semilattice $(X,\preceq)$\ can be metrized by
means of an equivalent radially convex metric.
\end{thm}

\begin{proof}
A well-known result that goes back to Vaughan \cite{Vag} says that every
locally compact and second-countable Hausdorff space can be metrized by a
boundedly compact metric so that every closed and bounded set in the space
is compact. Let $d$ be such a metric on $X,$ and note that $X$ is separable
(as it is second-countable). Let $\{x_{1},x_{2},...\}$ be a countable dense
set in $X,$ and define $\rho_{d}:\mathcal{C}_{0}(X)\times\mathcal{C}%
_{0}(X)\rightarrow[0,\infty)$ by 
\begin{equation*}
\rho_{d}(A,B):=\sum_{i\geq1}2^{-i}\min\{1,\left\vert
d(x_{i},A)-d(x_{i},B)\right\vert \}\text{.} 
\end{equation*}
It is well-known that $\rho_{d}$ metrizes the Wijsman topology on $\mathcal{C%
}_{0}(X)$ \cite[p. 37]{Be}. Moreover, bounded compactness of $d$ ensures
that the Fell topology and the Wijsman topology on $\mathcal{C}_{0}(X)$ are
the same \cite[Theorem 5.1.10]{Be}. We may thus conclude that the metric
topology induced by $\rho_{d}$ is the Fell topology on $\mathcal{C}_{0}(X).$
Moreover, $d(x,B)\leq d(x,A)$ for every $x\in X$ and $A,B\in\mathcal{C}%
_{0}(X)$ with $A\subseteq B.$ Consequently, for every $A,B,C\in\mathcal{C}%
_{0}(X)$ with $A\subseteq B\subseteq C,$ we have 
\begin{equation*}
d(x,A)-d(x,C)\geq d(x,B)-d(x,C)\hspace{0.2in}\text{for all }x\in X, 
\end{equation*}
and it follows that $\rho_{d}(A,C)\geq\rho_{d}(B,C).$ As one similarly shows
that $\rho_{d}(A,C)\geq\rho_{d}(A,B)$, it follows that $\rho_{d}$ is a
radially convex metric on $\mathcal{C}_{0}(X)$ relative to the containment
ordering $\subseteq$.

We now define $D:X\times X\rightarrow[0,\infty)$ by $D(x,y):=\rho_{d}(x^{%
\downarrow},y^{\downarrow})$. As the canonical order-embedding from $X$ into 
$\mathcal{C}_{0}^{\text{F}}(X)$ is a topological embedding (Theorem \ref%
{embed1}), it is plain that $D$ is a metric on $X$ that induces the topology
of $X.$ Moreover, as $\rho_{d}$ is radially convex relative to $\subseteq$,
and $x\preceq y$ iff $x^{\downarrow}\subseteq y^{\downarrow}$, $D$ is
radially convex relative to $\preceq$.
\end{proof}

\subsection{On Complete $\wedge$-Homomorphisms}

Let $(X_{1},\preceq_{1})$ and $(X_{2},\preceq_{2})$ be two $\wedge$\textit{-}%
semilattices. We recall that a map $f:X\rightarrow Y$ is \textit{%
order-preserving} if $x\preceq_{1}y$ implies $f(x)\preceq_{2}f(y),$ \textit{%
order-reversing} if $x\preceq_{1}y$ implies $f(y)\preceq_{2}f(x),$ and that
it is an $\wedge$\textit{-homomorphism} if $f(x\wedge y)=f(x)\wedge f(y)$
for every $x,y\in X_{1}.$ An $\wedge$-homomorphism is always
order-preserving, but not conversely. If both $(X_{1},\preceq_{1})$ and $%
(X_{2},\preceq_{2})$ are complete $\wedge$\textit{-}semilattices, and $%
f(\bigwedge S)=\bigwedge f(S)$ for every nonempty $S\subseteq X_{1},$ we say
that $f$ is a \textit{complete }$\wedge$\textit{-homomorphism. }

Where $(X,\preceq)$ is any poset, we follow Gierz et al. \cite{Dom} in
saying that a nonempty subset $S$ of $X$ is $\preceq$\textit{-filtered }if
every nonempty finite subset of $S$ has a $\preceq$-lower bound in $S$. For
instance, if restricting $\preceq$ to $S$ yields an $\wedge$\textit{-}%
semilattice, then $S$ is sure to be $\preceq$-filtered. In particular, for
every $\preceq$-decreasing sequence $\langle x_{n}\rangle$ in $X$, the set $%
\{x_{1},x_{2},...\}$ is $\preceq$-filtered.

In this application, we are interested in understanding when a map between
complete Hausdorff topological $\wedge$-semilattices preserve the infima of
filtered sets. This is not a trivial matter in that even when such a map $f$
is a continuous $\wedge$-homomorphism, and $\langle x_{n}\rangle$ is a
decreasing sequence in the domain of $f$, the equation $f(x_{1}\wedge
x_{2}\wedge\cdot\cdot\cdot)=f(x_{1})\wedge f(x_{2})\wedge\cdot\cdot\cdot$
may fail.\footnote{%
For any countable set $\{x_{1},x_{2},...\}$ in an $\wedge$-semilattice, we
write $x_{1}\wedge x_{2}\wedge\cdot\cdot\cdot$ to mean $\bigwedge%
\{x_{1},x_{2},...\}$.}

\begin{example}
Let $X:=\{-1\}\cup(0,1]$ and $Y:=\{-1\}\cup[0,1],$ and endow these sets with
the usual order and topology. This makes both of these sets locally compact
and totally bounded metric spaces that are also complete topological
lattices. Now define $f:X\rightarrow Y$ by $f(x):=x,$ which is of course a $%
\wedge$-homomorphism and a topological embedding. And yet $f(1\wedge\frac{1}{%
2}\wedge\frac{1}{3}\wedge\cdot\cdot\cdot)=-1$ while $f(1)\wedge f(\frac{1}{2}%
)\wedge\cdot\cdot\cdot=0.$
\end{example}

Our first theorem in this application shows that such counter-examples do
not arise so long as the complete topological semilattice in the domain is
locally compact \textit{and} order-connected. It may be worth noting that
this result imposes virtually no additional conditions on the target
complete topological semilattice.

\begin{thm}
\label{app2thm}Let $(X_{1},\preceq_{1})$\ and $(X_{2},\preceq_{2})$\ be two
topological posets that happen to be complete $\wedge$-semilattices. Suppose 
$(X_{1},\preceq_{1})$ is locally compact and order-connected, and $%
f:X_{1}\rightarrow X_{2}$ is an order-preserving and continuous map. Then, $%
f(\bigwedge S)=\bigwedge f(S)$\ for every $\preceq_{1}$-filtered subset\ $S$%
\ of\ $X_{1}.$
\end{thm}

In the proof of this result we shall utilize the notions of decreasing and
essentially decreasing nets in a poset. Let $(X,\preceq)$ be a poset and $%
\langle x_{\lambda}\rangle$ a net in $X$ with the underlying directed set $%
(\Lambda,\trianglelefteq).$ We say that $\langle x_{\lambda}\rangle$ is $%
\preceq$-\textit{decreasing} if $x_{\beta}\preceq x_{\alpha}$ holds for
every $\alpha,\beta\in\Lambda$ with $\alpha\trianglelefteq\beta$. More
generally, we say that $\langle x_{\lambda}\rangle$ is \textit{essentially }$%
\preceq$-\textit{decreasing} if for every $\alpha\in\Lambda$ there exists a $%
\beta\in\Lambda$ such that $x_{\lambda}\preceq x_{\alpha}$ holds for every $%
\lambda\in\Lambda$ with $\beta\trianglelefteq\lambda$.

The following proposition, which highlights another desirable property of
the Fell topology, will be used to prove the above theorem.

\begin{prop}
\label{app2prop}Let $X$ be a Hausdorff space and $\langle S_{\lambda}\rangle$
an essentially $\subseteq$-decreasing net in $\mathcal{C}(X)$ with the
underlying directed set $(\Lambda,\trianglelefteq).$ Then, $%
S_{\lambda}\rightarrow\bigcap\nolimits _{\lambda\in\Lambda}S_{\lambda}$ in
the Fell topology.
\end{prop}

\begin{proof}
We put $S:=\bigcap\nolimits _{\lambda\in\Lambda}S_{\lambda}$. Our goal is to
show that $\langle S_{\lambda}\rangle$ is eventually in any subbasic open
neighborhood of $S$ (relative to the Fell topology). To this end, notice
first that if $S\in O^{-}$ for some open set $O$ in $X,$ then $%
S_{\lambda}\cap O\supseteq S\cap O\neq\varnothing,$ whence $S_{\lambda}\in
O^{-}$, for each $\lambda.$ Next, suppose $S\in(X\backslash K)^{+}$ for some
compact $K\subseteq X.$ If $K=\varnothing,$ we trivially have $%
S_{\lambda}\in(X\backslash K)^{+}$ for each $\lambda,$ so assume instead $K$
is nonempty. As $K\cap S=\varnothing,$ for every $x\in K$ there is a $%
\lambda(x)\in\Lambda$ such that $x$ does not belong to $S_{\lambda(x)}$.
Then, $\{X\backslash S_{\lambda(x)}:x\in K\}$ is an open cover of $K,$ so
there is a finite $F\subseteq K$ such that $K\subseteq\bigcup\nolimits
_{x\in F}X\backslash S_{\lambda(x)}$, that is, $K\cap\bigcap\nolimits _{x\in
F}S_{\lambda(x)}=\varnothing$. But, as $\langle S_{\lambda}\rangle$ is
essentially $\subseteq$-decreasing, there exists a $\beta\in\Lambda$ such
that $S_{\lambda}\subseteq\bigcap\nolimits _{x\in F}S_{\lambda(x)}$, whence $%
K\cap S_{\lambda}=\varnothing$, for every $\lambda\in\Lambda$ with $%
\beta\trianglelefteq\lambda$. This proves that $\langle S_{\lambda}\rangle$
is eventually in $(X\backslash K)^{+},$ as desired.
\end{proof}

\noindent \textbf{Proof of Theorem \ref{app2thm}.} Let $S$ be a $\preceq_{1}$%
-filtered subset\textit{\ }of\textit{\ }$X_{1},$ and consider the directed
poset $(\mathcal{F}_{0}(S),\subseteq)$ where $\mathcal{F}_{0}(S)$ is the
family of all nonempty finite subsets of $S.$ Then, for every $A\in\mathcal{F%
}_{0}(S)$ there exists a $y_{A}\in S$ such that $y_{A}\preceq_{1}\bigwedge A$
(but note that $\bigwedge A$ need not lie in $S$). We note that the net $%
\langle y_{A}^{\downarrow}\rangle$ with the underlying directed set $(%
\mathcal{F}_{0}(S),\subseteq)$ is essentially $\subseteq$-decreasing.
Indeed, for any $A\in\mathcal{F}_{0}(S)$, letting $B:=A\cup\{y_{A}\}$ we
find $y_{C}\preceq_{1}\bigwedge C\preceq_{1}\bigwedge B\preceq_{1}y_{A}$,
whence $y_{C}^{\downarrow}\subseteq y_{A}^{\downarrow}$, for every $C\in%
\mathcal{F}_{0}(S)$ with $B\subseteq C$.

Now, it is easily verified that 
\begin{equation*}
\bigcap_{A\in\mathcal{F}_{0}(S)}y_{A}^{\downarrow}=\left(\bigwedge
S\right)^{\downarrow}\text{.} 
\end{equation*}
Consequently, by Proposition \ref{app2prop}, and because $\langle
y_{A}^{\downarrow}\rangle$ is essentially $\subseteq$-decreasing, we have $%
y_{A}^{\downarrow}\rightarrow\left(\bigwedge S\right)^{\downarrow}$ relative
to the Fell topology. By Theorem \ref{continv}, therefore, $%
y_{A}\rightarrow\bigwedge S$. By continuity of $f,$ then, $%
f(y_{A})\rightarrow f(\bigwedge S)$. Since $\bigwedge f(S)\preceq_{2}f(y_{A})
$ for each $A\in\mathcal{F}_{0}(S)\ $and $\preceq_{2}$ is closed, it follows
that $\bigwedge f(S)\preceq_{2}f(\bigwedge S)$. On the other hand, $%
f(\bigwedge S)\preceq_{2}\bigwedge f(S)$ simply because $f(\bigwedge S)$ is
a $\preceq_{2}$-lower bound for $f(S)$ (as $f$ is order-preserving). We
conclude that $f(\bigwedge S)=\bigwedge f(S)$. $\blacksquare$

\bigskip{}

The following observation provides a reflection of Theorem \ref{app2thm} for
essentially decreasing sequences.

\begin{cor}
Let $(X_{1},\preceq_{1})$, $(X_{2},\preceq_{2}),$\ and $f$ be as in Theorem %
\ref{app2thm}. Then, 
\begin{equation}
f(x_{1}\wedge x_{2}\wedge\cdot\cdot\cdot)=\lim f(x_{n})=f(x_{1})\wedge
f(x_{2})\wedge\cdot\cdot\cdot  \label{cc}
\end{equation}
for every essentially $\preceq_{1}$-decreasing sequence $\langle x_{n}\rangle
$ in $X_{1}.$
\end{cor}

\begin{proof}
Let $\langle x_{n}\rangle$ be an essentially $\preceq_{1}$-decreasing
sequence in $X_{1}.$ Then, $\{x_{1},x_{2},...\}$ is a $\preceq_{1}$-filtered
subset\textit{\ }of\textit{\ }$X_{1},$ so Theorem \ref{app2thm} entails that
the first and third expressions of (\ref{cc}) are equal. Moreover, $\langle
x_{n}^{\downarrow}\rangle$ is an essentially $\subseteq$-decreasing sequence
in $\mathcal{C}^{\text{F}}(X),$ so by Proposition \ref{app2prop}, $%
x_{n}^{\downarrow}\rightarrow\bigcap_{i\geq1}x_{i}^{\downarrow}=(x_{1}\wedge
x_{2}\wedge\cdot\cdot\cdot)^{\downarrow}.$ By Theorem \ref{continv},
therefore, $x_{n}\rightarrow x_{1}\wedge x_{2}\wedge\cdot\cdot\cdot,$ so by
continuity of $f,$ we get the equality of the first and second expressions
of (\ref{cc}).
\end{proof}

Theorem \ref{app2thm} does not go so far as to ensure that the subject map $f
$ is an $\wedge$-homomorphism. The following example shows that this is not
warranted by the hypotheses of the theorem even in the compact case.

\begin{example}
Let $X_{1}:=[-1,1]^{2}$ and $X_{2}:=[-2,2].$ Endowing these sets with the
usual order and topology yields order-connected and complete topological
lattices whose carriers are metric continua. Consider the map $%
f:X_{1}\rightarrow X_{2}$ with $f(x,y):=x+y,$ which is both order-preserving
and continuous. And yet, $f((0,1)\wedge(1,0))=0$ while $f(0,1)\wedge
f(1,0)=1.$
\end{example}

Having said this, we note that strengthening the order-preservation
hypothesis in Theorem \ref{app2thm} to being an $\wedge$-homomorphism
furnishes a complete $\wedge$-homomorphism.

\begin{thm}
\label{homo}Let $(X_{1},\preceq_{1})$\ and $(X_{2},\preceq_{2})$\ be as in
Theorem \ref{app2thm}. Then, every continuous $\wedge$-homomorphism $%
f:X_{1}\rightarrow X_{2}$ is complete.
\end{thm}

\begin{proof}
Let $S$ be any nonempty subset of $X_{1},$ and again denote by $\mathcal{F}%
_{0}(S)$ the collection of all nonempty finite subsets of $S.$ Define 
\begin{equation*}
T:=\left\{ \bigwedge A:A\in\mathcal{F}_{0}(S)\right\} 
\end{equation*}
which is the smallest sub-$\wedge$-semilattice of $(X_{1},\preceq_{1})$ that
contains $S.$ Note that, for any $A\in\mathcal{F}_{0}(S),$ we have $%
\bigwedge f(S)\preceq_{2}\bigwedge f(A)=f(\bigwedge A),$ because $A\subseteq
S$ and $f$ is an $\wedge$-homomorphism. Thus, $\bigwedge f(S)$ is a $%
\preceq_{2}$-lower bound for $f(T),$ whence $\bigwedge
f(S)\preceq_{2}\bigwedge f(T)$. But, obviously, $\bigwedge S=\bigwedge T,$
so by Theorem \ref{app2thm}, $f(\bigwedge S)=f(\bigwedge T)=\bigwedge f(T).$
It follows that $\bigwedge f(S)\preceq_{2}f(\bigwedge S).$ As $f(\bigwedge
S)\preceq_{2}\bigwedge f(S)$ holds simply because $f$ is order-preserving,
we are done.
\end{proof}

It is well-known, and is easily proved, that a decreasing net in a compact
Hausdorff topological $\wedge$-semilattice converges to its infimum. (See,
for instance, Strauss \cite{Strauss}.) Using this fact, one can show easily
that every compact Hausdorff topological $\wedge$-semilattice is a complete $%
\wedge$-semilattice. Moreover, again using this result, we can show that a
continuous $\wedge$-homomorphism between compact Hausdorff topological $%
\wedge$-semilattices $(X_{1},\preceq_{1})$\ and $(X_{2},\preceq_{2})$ is a
complete $\wedge$-homomorphism. Theorem \ref{homo} shows that we can replace
the compactness requirements on the domain and codomain in this result with
weaker complete $\wedge$-semilattice requirements, provided that $%
(X_{1},\preceq_{1})$ is locally compact and order-connected.

\subsection{A Fixed Point Theorem}

A poset $(X,\preceq)$ is said to satisfy the \textit{countable-chain
condition} if $\bigwedge S$ exists for every nonempty countable chain $S$ in 
$(X,\preceq).$ This condition plays an important role in a variety of
order-theoretic fixed point theorems. One of the most well-known of these is
the following:

\bigskip{}

\noindent \textbf{The Tarski-Kantorovich Fixed Point Theorem.}\footnote{%
See \cite[p. 15]{Du-Gr} and \cite[p. 630]{J}} \textit{Let }$(X,\preceq)$ 
\textit{be a poset that satisfies the countable-chain condition, and }$%
f:X\rightarrow X$ \textit{a function such that }$f(x_{1}\wedge
x_{2}\wedge\cdot\cdot\cdot)=f(x_{1})\wedge f(x_{2})\wedge\cdot\cdot\cdot$ 
\textit{for every }$\preceq$\textit{-decreasing sequence }$(x_{m})$\textit{\
in }$X$\textit{. If }$f(x)\preceq x$\textit{\ for some }$x\in X$\textit{,
then }$f$ \textit{has a fixed point.}

\bigskip{}

The following is a topological variant of this fixed point theorem that
devolves from our result about the continuity of the map $%
x^{\downarrow}\mapsto x$.

\begin{thm}
Let $(X,\preceq)$ be a locally compact and order-connected topological poset
that satisfies the countable-chain condition, and $f:X\rightarrow X$ an
order-preserving and continuous map. If $f(x)\preceq x$\ for some $x\in X$,
then $f$ has a fixed point.
\end{thm}

\begin{proof}
Take any point $x\in X$ with $f(x)\preceq x$, and define $x_{1}:=x$ and $%
x_{n}:=f(x_{n-1})$ for each $n=2,3,...$. As $f$ is order-preserving, $%
\left\langle x_{n}\right\rangle $ is a $\preceq$-decreasing sequence in $X.$
In particular, $x_{\ast}:=x_{1}\wedge x_{2}\wedge\cdot\cdot\cdot$ exists by
the countable-chain condition. Moreover, $\left\langle
x_{n}^{\downarrow}\right\rangle $ is a $\subseteq$-decreasing sequence in $%
\mathcal{C}^{\text{F}}(X),$ so Proposition 5.6 entails $x_{n}^{\downarrow}%
\rightarrow x_{1}^{\downarrow}\cap
x_{2}^{\downarrow}\cap\cdot\cdot\cdot=x_{\ast}^{\downarrow}.$ By Theorem \ref%
{continv}, therefore, $x_{n}\rightarrow x_{\ast}$, whence, by continuity of $%
f,$ we find $f(x_{n})\rightarrow f(x_{\ast}).$ As $\left\langle
f(x_{n})\right\rangle $ is the sequence $\left\langle x_{n+1}\right\rangle ,$
we must conclude that $f(x_{n})\rightarrow x_{\ast}$. Thus: $%
f(x_{\ast})=x_{\ast}$.
\end{proof}

Checking for order-preservation and continuity properties are in most
applications easier than checking for the homomorphism property required in
the Tarski-Kantorovich theorem. This is the advantage of Theorem 5.11.

\subsection{On Anderson's Theorem}

Let $(X,\preceq)$ be a poset whose carrier is a topological space. Such a
poset is said to be \textit{locally convex }if the topology of $X$ has a
basis that consists of $\preceq$-convex sets. A major source of locally
convex posets is the following celebrated result.

\bigskip{}

\noindent \textbf{Anderson's Theorem.} \textit{Every locally compact and
connected Hausdorff topological lattice is locally convex.}

\bigskip{}

Anderson's original proof for this result \cite[Theorem 1]{And} is fairly
long. By contrast, a straightforward proof for it obtains if we use the
second embedding theorem of Section 4. Indeed, Anderson's theorem is an
immediate consequence of Theorem \ref{embed2} and the following elementary
observation.

\begin{prop}
\label{And}For any Hausdorff space $X,$ both $(\mathcal{C}^{\text{F}%
}(X),\subseteq)$ and $(\mathcal{C}_{\downarrow}^{\text{F}}(X),\subseteq)$
are locally convex.
\end{prop}

\begin{proof}
By definition, a basis for the Fell topology on $\mathcal{C}(X)$ consists of
sets of the form $O_{1}^{-}\cap\cdot\cdot\cdot\cap O_{k}^{-}\cap(X\backslash
K)^{+}$ where $k\in\mathbb{N},$ $O_{1},...,O_{k}$ are open subsets of $X$
and $K$ a compact subset of $X.$ Now suppose $A$ and $B$ are two elements of
such a set. Then, $A\subseteq C$ implies $C\in O_{i}^{-}$ for each $%
i=1,...,k,$ and $C\subseteq B$ implies $C\in(X\backslash K)^{+}.$ It follows
that the standard basis for the Fell topology on $\mathcal{C}(X)$ consists
of $\subseteq$-convex sets, i.e., $(\mathcal{C}^{\text{F}}(X),\subseteq)$ is
locally convex. The second assertion follows from the first.
\end{proof}

There is more one can say about locally compact and connected Hausdorff
topological lattices. Let $(X,\preceq)$ be a lattice, and $\left\langle
x_{\lambda}\right\rangle $\ a net in $X.$ For any $x\in X$, recall that $%
\left\langle x_{\lambda}\right\rangle $\ \textit{order-converges} to $x$ if
there exist two nets $\left\langle y_{\lambda}\right\rangle $\ and $%
\left\langle z_{\lambda}\right\rangle $\ in $X$ such that (i) $\left\langle
y_{\lambda}\right\rangle $\ is $\preceq$-increasing and $\left\langle
z_{\lambda}\right\rangle $ is $\preceq$-decreasing; (ii) $y_{\lambda}\preceq
x_{\lambda}\preceq z_{\lambda}$\ for every $\lambda$; and (iii) $\bigvee
y_{\lambda}=x=\bigwedge z_{\lambda}$. Now let \textit{$\mathcal{B}$ }be the
collection of all $\preceq$-convex subsets $U$\ of $X$ such that for any net 
$\left\langle x_{\lambda}\right\rangle $\ in $X$ and any $x\in U$\ such that 
$\left\langle x_{\lambda}\right\rangle $ order-converges to $x,$ the net $%
\left\langle x_{\lambda}\right\rangle $\ is eventually in $U.$ It is readily
checked that $\mathcal{B}$ is a basis for a topology on $X$. Lawson \cite%
{Law2} refers to the topology generated by this basis as\textit{\ }the 
\textit{convex-order topology} on $X$.

\begin{prop}
Let $(X,\preceq)$ be a locally compact and connected Hausdorff topological
lattice. Then, the topology of $X$ is coarser than the convex-order topology.
\end{prop}

\begin{proof}
By Anderson's Theorem, the topology of $X$ has a basis that consists of $%
\preceq$-convex sets. Let $U$ be any such set, and take any $x\in U.$
Suppose $\left\langle x_{\lambda}\right\rangle $ is a net in $X$ that
order-converges to $x.$ Then, by definition, there exist two nets $%
\left\langle y_{\lambda}\right\rangle $ and $\left\langle
z_{\lambda}\right\rangle $ in $X$ with the properties (i), (ii) and (iii)
stated above. By (i), $\langle z_{\lambda}^{\downarrow}\rangle$ is a
decreasing net, so $z_{\lambda}^{\downarrow}\rightarrow\bigcap
z_{\lambda}^{\downarrow}$ in the Fell topology. Since $\bigcap
z_{\lambda}^{\downarrow}=\left(\bigwedge z_{\lambda}\right)^{\downarrow},$
(iii) entails that $z_{\lambda}^{\downarrow}\rightarrow x^{\downarrow}$ in
the Fell topology, whence $z_{\lambda}\rightarrow x$ by Theorem \ref{embed2}%
. Similarly, $\langle y_{\lambda}^{\uparrow}\rangle$ is a decreasing net, so 
$y_{\lambda}^{\uparrow}\rightarrow\bigcap
y_{\lambda}^{\uparrow}=\left(\bigvee
y_{\lambda}\right)^{\uparrow}=x^{\uparrow}$ in the Fell topology. Therefore,
applying Theorem \ref{embed2} in the context of the dual lattice $%
(X,\succeq),$ we find $y_{\lambda}\rightarrow x.$ It follows that both $%
\left\langle y_{\lambda}\right\rangle $ and $\left\langle
z_{\lambda}\right\rangle $ are eventually in $U.$ As $U$ is $\preceq$%
-convex, (ii) implies that $\left\langle x_{\lambda}\right\rangle $ is
eventually in $U$. In view of the arbitrary choice of $x$ in $U,$ we
conclude that $U$ is open in the convex-order topology.
\end{proof}

In passing, we should note that this is only one-half of the story. In
general, it may or may not be possible to topologize a given lattice $%
(X,\preceq)$ to obtain a locally compact and connected Hausdorff topological
lattice. But Lawson \cite[Theorem 10]{Law2} shows that if there is such a
topology, then it must be finer than the convex-order topology. Combining
this with the proposition above thus establishes Lawson's surprising
finding: \textit{There is at most one topology on a lattice that would make
it a locally compact and connected Hausdorff topological lattice.}

\subsection{On Completely Order-Regular Posets}

A topological poset $(X,\preceq)$ is said to be a \textit{completely regular
ordered space} if

\begin{enumerate}
\item for every $x,y\in X$ such that $x\preceq y$ is false, there exists a
continuous and order-preserving $f:X\rightarrow\mathbb{R}$ with $f(x)>f(y);$
and

\item for every $x\in X$ and a neighborhood $V$ of $x$ in $X,$ there are
continuous functions $f:X\rightarrow[0,1]$ and $g:X\rightarrow[0,1]$ such
that $f$ is order-preserving, $g$ is order-reversing, $f(x)=1=g(x),$ and $%
\min\{f(z),g(z)\}=0$ for all $z\in X\backslash V$.
\end{enumerate}

\noindent This concept, which was introduced by Nachbin \cite{Na},
generalizes the topological notion of Tychonoff space. (For, $(X,=)$ is a
completely regular ordered space iff $X$ is completely regular.) Due to its
intimate connection to the theory of order-compactifications, completely
regular ordered spaces have received quite a bit of attention in the
literature. (See, among others, \cite{B-M,B-F,ChoePark,Na,Sal}.)

A topological poset $(Y,\trianglelefteq)$ is said to be an \textit{%
order-compactification} of $(X,\preceq)$ if $Y$ is compact, and there exists
a topological order-embedding from $X$ into $Y$ whose range is dense in $Y.$
A well-known theorem of Nachbin \cite{Na} says that $(X,\preceq)$ admits an
order-compactification if, and only if, it is a completely regular ordered
space. This fact alone motivates finding conditions for a topological poset
to qualify as a completely regular ordered space. In this application, we
use our embedding theorem to provide such conditions.

Let $(X,\preceq)$ be a locally compact and connected Hausdorff topological
lattice. By Theorem \ref{embed2}, $x\mapsto x^{\downarrow}$ is a
homeomorphism from $X$ onto $\mathcal{C}_{\downarrow}^{\text{F}}(X),$ and of
course, it is an order-isomorphism between $(X,\preceq)$ and $(\mathcal{C}%
_{\downarrow}(X),\subseteq).$ Furthermore, cl$(\mathcal{C}_{\downarrow}(X)),$
where the closure is taken with respect to the Fell topology, is a compact
subspace of $\mathcal{C}^{\text{F}}(X),$ because $\mathcal{C}^{\text{F}}(X)$
is itself a compact space. Thus, by Proposition \ref{whenposet}, $($cl$(%
\mathcal{C}_{\downarrow}(X)),\subseteq)$ is a compact topological poset, and
as such, it is an order-compactification of $(X,\preceq)$. In view of
Nachbin's said compactification theorem, therefore:

\begin{thm}
\label{app3thm}Every locally compact and connected Hausdorff topological
lattice is a completely regular ordered space.
\end{thm}

This observation is a companion to a related result of Burgess and
Fitzpatrick \cite{B-F}. To see this, we follow Priestley \cite{Pri} in
referring to a topological poset $(X,\preceq)$ as an $I$\textit{-space} if $%
O^{\uparrow}$ and $O^{\downarrow}$ are open subsets of $X$ for every open $%
O\subseteq X.$ In this terminology, Corollary 4.5 of \cite{B-F} entails that 
\textit{every locally compact and locally convex }$I$\textit{-space }$%
(X,\preceq)$ \textit{is a completely regular ordered space}. Therefore, in
view of Anderson's Theorem, Theorem \ref{app3thm} replaces the properties of
local convexity and being an $I$-space in the noted result of \cite{B-F}
with connectedness in the context of locally compact Hausdorff topological
lattices.

\section{Canonical Order-Embedding of a Topological Po-Group}

In this section, we show that one can replace the lattice structure in
Theorem \ref{embed1} with an alternative algebraic structure that is
compatible with the ordering of the ambient poset.

By a \textit{po-group}, we mean an ordered triple $(X,\cdot,\preceq)$ where $%
(X,\cdot)$ is a group whose law of composition is written multiplicatively, $%
(X,\preceq)$ is a poset, and $\preceq$ is translation-invariant relative to $%
\cdot,$ that is, $xz\preceq yz$ and $zx\preceq zy$ for every $x,y,z\in X$
with $x\preceq y.$ We denote the identity of $(X,\cdot)$ by $\mathbf{1},$
and adopt the standard notation for \textit{Minkowski product}. That is, $%
A\cdot B:=\{ab:(a,b)\in A\times B\}$ for any nonempty $A,B\subseteq X,$ but
as usual, we write $x\cdot A$ instead of $\{x\}\cdot A,$ and similarly $%
A\cdot x:=A\cdot\{x\},$ for any $x\in X$. In the context of a po-group $%
(X,\cdot,\preceq),$ we have $x\preceq y$ iff $xy^{-1}\in\mathbf{1}%
^{\downarrow}$ iff $y^{-1}x\in\mathbf{1}^{\downarrow},$ so the partial order 
$\preceq$ is entirely determined by the principal ideal $\mathbf{1}%
^{\downarrow}$. In particular, $x\cdot\mathbf{1}^{\downarrow}=x^{\downarrow}=%
\mathbf{1}^{\downarrow}\cdot x$ for any $x\in X.$

For any $x,y\in X,$ if $z\preceq x$ and $w\preceq y,$ then by translation
invariance of $\preceq$ we get $zw\preceq zy\preceq xy,$ so $%
zw\in(xy)^{\downarrow}.$ Thus, $x^{\downarrow}\cdot
y^{\downarrow}\subseteq(xy)^{\downarrow}$. Conversely, if $z\preceq xy,$
then $w:=z(xy)^{-1}\preceq\mathbf{1},$ so 
\begin{equation*}
z=w(xy)=(wx)y\in x^{\downarrow}\cdot y\subseteq x^{\downarrow}\cdot
y^{\downarrow}\text{.} 
\end{equation*}
Conclusion: $x^{\downarrow}\cdot y^{\downarrow}=(xy)^{\downarrow}$. Using
this observation, we find readily that $(\mathcal{C}_{\downarrow}(X),\cdot)$
is a group where the law of composition is the Minkowski product, the
identity is $\mathbf{1}^{\downarrow}$, and $(x^{\downarrow})^{-1}=(x^{-1})^{%
\downarrow}$ foe each $x\in X.$ Besides, if $x^{\downarrow}\subseteq
y^{\downarrow}$, then $x^{\downarrow}\cdot z^{\downarrow}=(x\cdot
z)^{\downarrow}\subseteq(y\cdot z)^{\downarrow}=y^{\downarrow}\cdot
z^{\downarrow}$ by translation-invariance of $\preceq$, and similarly, $%
z^{\downarrow}\cdot x^{\downarrow}\subseteq z^{\downarrow}\cdot
y^{\downarrow}$. It follows that $(\mathcal{C}_{\downarrow}(X),\cdot,%
\subseteq)$ is a po-group, and obviously, $x\mapsto x^{\downarrow}$ is a
group isomorphism, as well as an order-isomorphism, from $(X,\cdot,\preceq)$
onto this po-group.

A \textit{topological po-group }is a po-group $(X,\cdot,\preceq)$ such that $%
(X,\cdot)$ is a topological group and $(X,\preceq)$ is a topological poset.
The rich structure of such po-groups allows us to sharpen the continuity
theorems of Section 4. In particular, the situation for the continuity of
the canonical order-embedding improves markedly in this framework.

\begin{prop}
\label{po-group1}Let $(X,\cdot,\preceq)$ be a topological po-group. Then,
the map $x\mapsto x^{\downarrow}$ from $X$ into $\mathcal{C}^{\text{F}}(X)$
is continuous.
\end{prop}

\begin{proof}
Take any $x\in X$ and any open subset $O$ of $X\ $with $x^{\downarrow}\in
O^{-}$. Then, $x^{\downarrow}\cap O\neq\varnothing$ which means there is a $%
z\in O$ with $zx^{-1}\preceq\mathbf{1}.$ We define $U:=xz^{-1}\cdot O$ which
is an open neighborhood of $x.$ Clearly, if $y\in U$, then $y=(xz^{-1})a$
for some $a\in O,$ so $a=(zx^{-1})y\preceq y,$ whence $y^{\downarrow}$
intersects $O$. Thus: $\{y^{\downarrow}:y\in U\}\subseteq O^{-}$.

Next, take any compact subset $K$ of $X$ such that $x^{\downarrow}\in(X%
\backslash K)^{+}$. Then, $x^{\downarrow}\cap K=\varnothing,$ so, since $%
x^{\downarrow}$ is a closed and $K$ is compact, a standard result of the
theory of topological groups says that there exists an open neighborhood $V$
of $\mathbf{1}$ in $X$ with $(x^{\downarrow}\cdot V)\cap K=\varnothing$. But 
$(x\cdot V)^{\downarrow}\subseteq x^{\downarrow}\cdot V,$ because if $%
z\preceq xv$ for some $v\in V,$ then for $w:=z(xv)^{-1},$ we have $w\preceq%
\mathbf{1},$ hence $wx\preceq x$, and we thus find $z=w(xv)=(wx)v\in
x^{\downarrow}\cdot V$. Consequently, $(x\cdot V)^{\downarrow}\cap
K=\varnothing,$ that is, $\{z^{\downarrow}:z\in x\cdot
V\}\subseteq(X\backslash K)^{+}$.
\end{proof}

We next combine Proposition \ref{po-group1} with Theorem \ref{continv} to
obtain our third embedding theorem.

\begin{thm}
\label{po-group2}Let $(X,\cdot,\preceq)$ be a locally compact and
order-connected topological po-group. Then, the map $x\mapsto x^{\downarrow}$
topologically order-embeds $(X,\preceq)$ in $\mathcal{C}^{\text{F}}(X)$.
\end{thm}

A \textit{partially ordered topological linear space} $X$ is a Hausdorff
topological (real) linear space $X$ which is endowed with a
translation-invariant closed partial order $\preceq$ that satisfies $\lambda
x\preceq\lambda y$ for every $x,y\in X$ and $\lambda\geq0$. Such a space is
an Abelian topological po-group relative to the addition operation of the
space. Moreover, its order-intervals are convex in the geometric sense, and
hence connected. Finally, if it is finite-dimensional, it is locally
compact. Thus, as an immediate consequence of Theorem \ref{po-group2}, we
get:

\begin{cor}
\label{po-group3}Let $X$ be a finite-dimensional partially ordered
topological linear space. Then the canonical order-embedding is a
topological embedding from $X$ into $\mathcal{C}^{\text{F}}(X)$.
\end{cor}

The following are two simple applications of these results.

\begin{example}
Let $(X,\cdot,\preceq)$ be a locally compact, second-countable and
order-connected topological po-group. Then, there is a radially convex
metric on $X$ that metrizes the topology of $X.$ This is proved exactly as
we proved Theorem \ref{U-C} but this time using Theorem \ref{po-group2}
instead of Theorem \ref{embed1}.\footnote{%
An interesting question at this junction is if one can choose this metric
left-invariant as well as radially convex. We do not know the answer to this
at present.}
\end{example}

\begin{example}
Let Sym$(n)$ stand for the set of all $n\times n$ symmetric real matrices.
We view this set as a linear space relative to the usual matrix addition and
scalar multiplication operations. Recall that a matrix $A\in$ Sym$(n)$ is
said to be \textit{positive semidefinite} if $\left\langle Ax,x\right\rangle
\geq0$ for every (real) $n$-vector $x$. In turn, the \textit{Loewner order} $%
\succcurlyeq_{\text{L}}$ is the partial order on Sym$(n)$ defined by $%
A\succcurlyeq_{\text{L}}B$ iff $A-B$ is positive semidefinite. Then, $($Sym$%
(n),\succcurlyeq_{\text{L}})$ is a finite-dimensional partially ordered
topological linear space (where the topology is naturally inherited from the
usual topology of $\mathbb{R}^{n\times n}$). By Corollary \ref{po-group3},
therefore, we may embed this space, both topologically and
order-theoretically, in $\mathcal{C}^{\text{F}}($Sym$(n))$.

We should note that $($Sym$(n),\succcurlyeq_{\text{L}})$ is not an $\wedge$%
-semilattice. Indeed, a well-known theorem of Kadison \cite{Kadison} says
that $($Sym$(n),\succcurlyeq_{\text{L}})$ is an antilattice, that is, the
greatest lower bound of two symmetric $n\times n$ matrices with respect to
the Loewner order exists iff these matrices are $\succcurlyeq_{\text{L}}$%
-comparable. Thus, this example is not covered by any of our earlier
embedding theorems.
\end{example}

A natural question at this point is if the rich setting of topological
po-groups allows us to delete either the local compactness or the
order-connectedness hypotheses from the statement of Theorem \ref{po-group2}%
. We conclude with two examples that demonstrate that this is not so, even
in the Abelian case.

\begin{example}
\label{Hilbertexample}Let $\mathbf{0}$ stand for the zero sequence, and let $%
\{e_{1},e_{2},...\}$ be the standard orthonormal base for the Hilbert space $%
\ell_{2}$. We endow $\ell_{2}$ with the coordinatewise order $\preceq$ (that
is, $x\preceq y$ iff $x_{i}\leq y_{i}$ for each $i=1,2,...$). Relative to
this order and its usual norm-topology, $\ell_{2}$ is a Hausdorff
topological lattice. (It is, in fact, a Banach lattice.) As such $\ell_{2}$
has convex order-intervals, so it is order-connected, but of course, it is
not locally compact. We claim that $e_{n}^{\downarrow}\overset{\text{K-P}}{%
\rightarrow}\mathbf{0}^{\downarrow}$. To see this, first, take any $x\in%
\mathbf{0}^{\downarrow},$ and note that $x=e_{n}+(x-e_{n})\in e_{n}+\mathbf{0%
}^{\downarrow}=e_{n}^{\downarrow}$, so the $n$th term of the constant
sequence $\left\langle x,x,..,\right\rangle $ is in $e_{n}^{\downarrow}$ for
each $n\in\mathbb{N},$ and this sequence obviously converges to $x.$ Second,
let $\langle n_{k}\rangle$ be a strictly increasing sequence of positive
integers, and take any convergent sequence $\langle x_{n_{k}}\rangle$ with $%
x_{n_{k}}\in e_{n_{k}}^{\downarrow}$ for each $k$. We need to show that $%
x:=\lim x_{n_{k}}$ belongs to $\mathbf{0}^{\downarrow}$, that is, $x_{i}\leq0
$ for each $i\in\mathbb{N}$. To this end, fix an arbitrary positive integer $%
i.$ Whenever $k>i,$ we have $n_{k}>i,$ so $x_{n_{k},i}\leq e_{n_{k},i}=0.$
But $\left\Vert x_{n_{k}}-x\right\Vert _{2}\rightarrow0$ implies $%
x_{n_{k},j}-x_{j}\rightarrow0$ for each $j\in\mathbb{N},$ so it follows that 
$x_{i}\leq0$. In view of the arbitrary choice of $i,$ we thus find $x\in%
\mathbf{0}^{\downarrow}$, as desired. Conclusion: $e_{n}^{\downarrow}\overset%
{\text{K-P}}{\rightarrow}\mathbf{0}^{\downarrow}$, so $e_{n}^{\downarrow}%
\rightarrow\mathbf{0}^{\downarrow}$ in the Fell topology. But, of course, $%
\left\langle e_{n}\right\rangle $ does not converge to $\mathbf{0}.$
\end{example}

We remark that $(\mathcal{C}_{\downarrow}^{\text{F}}(\ell_{2}),+)$ is not a
topological group: $x^{\downarrow}\mapsto(-x)^{\downarrow}$ is not a
continuous self-map on $\mathcal{C}_{\downarrow}^{\text{F}}(\ell_{2})$.
Indeed, while $e_{n}^{\downarrow}\rightarrow\mathbf{0}^{\downarrow}$ in the
Fell topology, the open ball around $\mathbf{0}$ of radius $\frac{1}{2}$
fails to hit $(-e_{n})^{\downarrow}$ for any $n\in\mathbb{N},$ so $%
(-e_{n})^{\downarrow}$ does not converge to $\mathbf{0}^{\downarrow}$ in $%
\mathcal{C}_{\downarrow}^{\text{F}}(\ell_{2})$. On the other hand, if $%
(X,+,\preceq)$ is a topological po-group that satisfies the conditions of
Theorem \ref{po-group2}, then it is plain that $(\mathcal{C}_{\downarrow}^{%
\text{F}}(X),+,\subseteq)$ is such a po-group as well.

\begin{example}
Let $X$ stand for the set of all square-summable $\mathbb{Z}$-valued
sequences. We view this set as a topological subgroup of $\ell_{2}$, and
endow it with the coordinatewise order $\preceq$. It is plain that all but
finitely many terms of any element of $X$ are zero. Thus, the topology that $%
X$ inherits from $\ell_{2}$ is discrete, and hence, locally compact.
However, as the only connected subsets of $X$ are the singletons, $%
(X,\preceq)$ is not order-connected. And sure enough, the inverse of the
order-embedding $x\mapsto x^{\downarrow}$ is not continuous on $\mathcal{C}%
_{\downarrow}^{\text{F}}(X).$ Indeed, we have $e_{n}^{\downarrow}\overset{%
\text{K-P}}{\rightarrow}\mathbf{0}^{\downarrow}$ but $\left\langle
e_{n}\right\rangle $ does not converge to $\mathbf{0}.$
\end{example}

\section{An Open Problem}

In this paper we looked at when a given topological poset $(X,\preceq)$ can
be embedded in the Fell hyperspace $\mathcal{C}^{\text{F}}(X)$ topologically
as well as order-theoretically. While our results identify a fairly rich
class of topological posets for which this is possible, in our embedding
theorems we worked exclusively with the canonical order-embedding $x\mapsto
x^{\downarrow}.$ This is certainly a natural starting point, for, as $%
x\mapsto x^{\downarrow}$ is always an order-embedding, it reduces the
problem to merely studying the bicontinuity of this map. However, one is
likely to obtain other embedding results by considering alternative maps
from $X$ into $\mathcal{C}(X).$ (Indeed, in Example \ref{start} we have
observed that the required embedding can be found even when $x\mapsto
x^{\downarrow}$ fails to be continuous.) Pursuing this direction, and
characterizing those topological posets that can be topologically
order-embedded in $\mathcal{C}^{\text{F}}(X),$ is left for future research.

\bigskip{}

\noindent \textsc{Acknowledgement }We thank professors Gerhard Gierz, Jinlu
Li and Nik Weaver for their insightful comments in the development stage of
this work.

\bigskip{}

\end{document}